\documentclass[11.5pt]{article}
\usepackage{floatrow}
\newfloatcommand{capbtabbox}{table}[][\FBwidth]
\usepackage{enumerate}
\usepackage{amsmath,cite,colordvi,xcolor}
\usepackage{amsthm}
\usepackage{amssymb}
\usepackage{amsfonts}
\usepackage{latexsym}
\usepackage{tikz}
\usepackage{float}
\usepackage{mathrsfs}[12pt]
\usepackage{multirow,multicol,subfigure}
\usepackage[colorlinks,urlcolor=blue,linkcolor=red,anchorcolor=green,citecolor=blue]{hyperref}
\usepackage{makecell}

\newtheorem{thm}{Theorem}[section]

\newtheorem{lem}[thm]{Lemma}

\theoremstyle{definition}

\theoremstyle{remark}

\def\sgn{{\rm sgn}}
\numberwithin{equation}{section}

\makeatletter \@addtoreset{equation}{section} \makeatother

\setbox0=\hbox{$+$}
\newdimen\plusheight
\plusheight=\ht0
\def\+{\;\lower\plusheight\hbox{$+$}\;}

\setbox0=\hbox{$\cdots$}
\newdimen\plusheight
\plusheight=\ht0
\def\dos{\;\lower\plusheight\hbox{$\cdots$}\;}

\allowdisplaybreaks[4]

\title{\bf  Finding
Modular Functions for Ramanujan-Type Identities}
\author{William Y.C. Chen$^1$, Julia Q.D. Du$^2$
and Jack C.D. Zhao$^3$
\date{$^{1,2}$Center for Applied Mathematics\\
Tianjin University\\
Tianjin 300072, P. R. China\\[12pt]
$^3$Center for Combinatorics\\
Nankai University\\
Tianjin 300071, P. R. China
\vskip 0.2 cm
Emails: $^1$chenyc@tju.edu.cn,
$^2$qddu@tju.edu.cn,
$^3$cdzhao@mail.nankai.edu.cn
\vskip 0.5 cm
\emph{{Dedicated to Professor George E. Andrews on the occasion of
his 80th birthday}}}
}

\begin{document}

\maketitle

\noindent{\bf Abstract.} This paper is concerned with a class of partition functions $a(n)$ introduced by Radu and defined in terms of eta-quotients.
By utilizing the transformation laws of Newman, Schoeneberg and Robins,
and  Radu's algorithms, we present an algorithm to find Ramanujan-type identities for $a(mn+t)$.
While this algorithm  is not guaranteed to succeed,
it applies to many cases.
For example,
we deduce a witness identity for $p(11n+6)$ with integer coefficients.
Our algorithm also leads to
Ramanujan-type identities for the overpartition functions $\overline{p}(5n+2)$ and $\overline{p}(5n+3)$ and Andrews--Paule's broken $2$-diamond partition
functions $\triangle_{2}(25n+14)$ and $\triangle_{2}(25n+24)$.
It can also be extended to derive
Ramanujan-type identities on a more general class of
partition functions. For example,
it yields the Ramanujan-type identities on Andrews'
singular overpartition functions
$\overline{Q}_{3,1}(9n+3)$ and $ \overline{Q}_{3,1}(9n+6)$ due to Shen,
the $2$-dissection formulas of Ramanujan  and the $8$-dissection formulas due to Hirschhorn.



\section{Introduction}

Throughout this paper, we follow the standard $q$-series notation in \cite{Gasper-Rahman-2004}:
\begin{align*}
 (a; q)_\infty=\prod\limits_{n=0}^\infty (1-aq^n)\quad
\text{and}
\quad (a_1,a_2, \ldots, a_m;q)_\infty=\prod\limits_{j=1}^m(a_j;q)_\infty,
\end{align*}
where $|q|<1$.
In the study of congruence properties and identities on partition functions,
Radu \cite{Radu-Thesis, Radu-2009, Radu-2015} defined a class of
partition functions $a(n)$ by
\begin{align}\label{radu-gf}
\sum_{n=0}^{\infty}
a(n)q^n=\prod_{\delta | M}
(q^\delta;q^\delta)^{r_\delta}_\infty,
\end{align}
where $M$ is a positive integer and ${r_\delta}$ are integers.
Many partition functions fall into the
framework of the above definition of $a(n)$, such as
the partition function $p(n)$,
the overpartition function $\overline{p}(n)$ \cite{Corteel-Lovejoy}, the Ramanujan $\tau$-function \cite{Ramanujan-1916,Hardy-1938,Hardy-1927},
the $k$-colored partition functions,
the $t$-core partition functions,
the $2$-colored Frobenius partition functions and the broken $k$-diamond partition
functions $\Delta_{k}(n)$ \cite{Andrews-Paule-2007}.

In this paper, we aim to present an algorithm to compute the generating function
\begin{align}\label{def-or}
\sum_{n=0}^\infty a(mn+t)q^n,
\end{align}
for fixed $m>0$ and $0\le t\le m-1$
by finding suitable  modular functions for $\Gamma_1(N)$.
When $M=1$ and $r_1=-1$,
$a(n)$ specializes to the partition function $p(n)$.
Kolberg \cite{Kolberg-1957} proved that for a positive integer $m$
prime to $6$, and $0\leq t\leq m-1$,
\begin{align}\label{Kolberg-relation}
\sum_{n=0}^\infty p(mn+t)q^{mn+t}  = (-1)^{(m-1)t} \frac{(q^{m^2}; q^{m^2})_\infty}{(q^m; q^m)^{m+1}_\infty} \det {M_t},
\end{align}
where $M_t=(g_{-t-i+j})_{(m-1)\times (m-1)}$,
\begin{align*}
g_t = \sum_{\frac{1}{2}n(3n+1) \equiv t(\text{mod }m)} (-1)^n q^{\frac{1}{2}n(3n+1)},
\end{align*}
and $g_t = g_s$ when $t\equiv s \pmod{m}$.
In view of \eqref{Kolberg-relation},
he derived some identities on $p(n)$, for example,
\begin{align}\label{GF-P-5n}
\sum_{n=0}^\infty p(5n)q^n=
\frac{(q^5; q^5)_\infty}{(q; q)_\infty^2(q, q^4; q^5)_\infty^8}
-3q\frac{(q^5; q^5)_\infty^6(q, q^4; q^5)_\infty^2}{(q; q)_\infty^7},
\end{align}
and
\begin{align*}
\left(\sum\limits_{n=0}^\infty p(5n)q^n\right)\left(\sum\limits_{n=0}^\infty p(5n+3)q^n\right)=3\frac{(q^5; q^5)_\infty^4}{(q; q)_\infty^6}+25q \frac{(q^5; q^5)_\infty^{10}}{(q; q)_\infty^{12}}.
\end{align*}
Atkin and Swinnerton-Dyer \cite{Atkin-Swinnerton-Dyer-1954}
have shown that $g_t$ can always
be expressed by certain infinite products for $m>3$.
Then the left hand side of \eqref{Kolberg-relation}
can be expressed in terms of certain infinite products.
Kolberg pointed out that when $m>5$, this becomes much more complicated.
For $m=11, 13$,
Bilgici and Ekin \cite{Bilgici-Ekin-2014-13, Bilgici-Ekin-2014}
used the method of Kolberg
to compute the generating function
\[\sum_{n=0}^\infty p(mn+t)q^{mn+t}\]
for all $0\leq t\leq m-1$.

Based on the ideas of Rademacher \cite{Rademacher-1942},
Newman \cite{Newman-1957, Newman-1959}
and  Kolberg \cite{Kolberg-1957},
Radu \cite{Radu-2009} developed an algorithm to verify the congruences
\begin{align}\label{a-m-t-cong}
a(mn+t)\equiv 0\pmod u,
\end{align}
for any given $m$, $t$ and $u$, and for all $n\geq 0$,
where $a(n)$ is defined in \eqref{radu-gf}.
Moreover, Radu \cite{Radu-2015} developed an algorithm,  called the Ramanujan--Kolberg algorithm, to derive
identities on the generating functions of $a(mn+t)$ using modular functions for $\Gamma_0(N)$.
A description of the Ramanujan--Kolberg algorithm can  be found in Paule and Radu \cite{Peter-Radu-2017-2}.
Smoot \cite{Smoot-2019} developed a Mathematica package \texttt{RaduRK} to implement Radu's algorithm.
It should be mentioned that
Eichhorn \cite{Eichhorn-1999} extended the technique
in \cite{Eichhorn-Ono-1995, Eichhorn-Sellers-2002} to partition functions $a(n)$ defined by
\begin{align}\label{Eichhorn-def}
\sum_{n=0}^\infty a(n)q^n = \prod_{j=1}^{L} (q^j; q^j)_\infty^{e_j},
\end{align}
where $L$ is a positive integer and $e_j$ are integers,
and reduced the verification of the congruences \eqref{a-m-t-cong}
to a finite number of cases.
It is easy to see that the defining relations
\eqref{radu-gf} and \eqref{Eichhorn-def} are equivalent to each other.
In this paper, we shall adopt the form of \eqref{radu-gf} in accordance with
the notation of eta-quotients.

Recall that the Dedekind eta-function $\eta(\tau)$ is defined by
\[\eta(\tau)=q^{\frac{1}{24}}\prod_{n=1}^\infty(1-q^n),\]
where $q=e^{2\pi i \tau}$, $\tau\in \mathbb{H}=\{\tau\in\mathbb{C}\colon{\rm Im}(\tau)>0\}$. An eta-quotient is a function of the form
\begin{align*}
\prod_{\delta|M}\eta^{r_{\delta}}(\delta \tau),
\end{align*}
where $M\geq 1$ and each $r_{\delta}$ is an integer.

The Ramanujan--Kolberg algorithm leads to verifications of some identities
on $p(n)$ due to Ramanujan \cite{Ramanujan-1919}, Zuckerman \cite{Zuckerman-1939}
and Kolberg \cite{Kolberg-1957}, for example,
\begin{align*}
\sum_{n=0}^\infty p(5n+4)q^n = 5\frac{(q^5; q^5)^5_\infty}{(q; q)^6_\infty},
\end{align*}
see \cite[eq. (18)]{Ramanujan-1919}.
It should be noted that there are
some Ramanujan-type identities that are not
covered by the Ramanujan--Kolberg algorithm,
such as the identity \eqref{GF-P-5n}.

In this paper, we  develop an algorithm to derive
Ramanujan-type identities for $a(mn+t)$ for $m>0$ and $0\leq t \leq m-1$,
which is essentially a modified version of Radu's algorithm.
We first find a necessary and sufficient condition for
a product of a generalized eta-quotient and the generating function
\eqref{def-or} to
be a modular function for $\Gamma_1(N)$
up to a power of $q$.
Then we try to express this modular function as a linear combination
of generalized eta-quotients over $\mathbb{Q}$.

For example, our algorithm leads to a verification  of
\eqref{GF-P-5n} for $p(5n)$. Moreover, we obtain Ramanujan-type identities for the overpartition functions $\overline{p}(5n+2)$ and $\overline{p}(5n+3)$ and the broken $2$-diamond partition
functions $\Delta_{2}(25n+14)$ and $\Delta_{2}(25n+24)$.
We also obtain the following witness identity with integer coefficients for $p(11n+6)$.

\begin{thm}\label{thm-p(11n+6)}
We have
\begin{align}
&z_0
\sum\limits_{n=0}^\infty p(11n+6)q^n \nonumber\\[6pt]
&\quad=
11  z^{10}+121   z^{8}e+330   z^{9}-484   z^7e-990   z^{8}+484   z^6 e+792   z^7\nonumber\\[6pt]
&\quad\qquad -484   z^5e+44   z^6+1089   z^4 e-132   z^5-1452   z^3e-451   z^4\nonumber\\[6pt]
&\qquad\quad +968   z^2e+748   z^3-242   ze-429   z^2+77   z+11,\label{ge-p-11-q}
\end{align}
where
\begin{align}
z_0 &= \frac{(q;q)_{\infty}^{24}}{q^{20} (q^{11};q^{11})_{\infty}^{23}
(q,q^{10};q^{11})_{\infty}^{28}
(q^2,q^9;q^{11})_{\infty}^{16}(q^3,q^{8};q^{11})_{\infty}^{12}
(q^4,q^{7};q^{11})_{\infty}^4},\nonumber\\[6pt]
z&=\frac{(q;q)_{\infty}}{q^{2}(q^{11};q^{11})_{\infty}(q,q^{10};q^{11})_{\infty}^{3}(q^2,q^9;q^{11})_{\infty}^{2}},
\label{z}\\[6pt] e&=\frac{(q;q)_{\infty}^{3}}{q^{3}(q^{11};q^{11})_{\infty}^{3}(q,q^{10};q^{11})_{\infty}^{5}(q^2,q^9;q^{11})_{\infty}^{5}(q^3,q^{8};q^{11})_{\infty}^{4}
(q^4,q^{7};q^{11})_{\infty}}.\label{e}
\end{align}
\end{thm}

Bilgici and Ekin \cite{Bilgici-Ekin-2014} deduced
a witness identity for $p(11n+6)$ with integer coefficients
using the method of Kolberg.
Radu \cite{Radu-2015} obtained a witness identity for
$p(11n+6)$ by using the Ramanujan--Kolberg algorithm.
Hemmecke \cite{Hemmecke-2018} generalized Radu's algorithm and derived a witness identity for  $p(11n+6)$.
Paule and Radu \cite{Paule-Radu-2017-witness}
found a polynomial relation on the generating function of $p(11n+6)$,
which can also be viewed as a witness identity.
Moreover, Paule and Radu \cite{Paule-Radu-2019-congruence}
found a witness identity for
$p(11n+6)$ in terms of  eta-quotients and the $U_2$-operator acting on
eta-quotients.

Our algorithm can be extended to a
more general class of partition functions $b(n)$ defined by
\begin{align}\label{def}
\sum_{n=0}^{\infty}
b(n)q^n=\prod_{\delta | M}
(q^\delta;q^\delta)^{r_\delta}_\infty
\prod_{\delta|M\atop 0<g<\delta}
(q^g,q^{\delta-g};q^\delta)^{r_{\delta,g}}_\infty,
\end{align}
where $M$ is a positive integer and
$r_\delta$, $r_{\delta, g}$ are integers.
Notice that \eqref{def} is a generalized eta-quotient
up to a power of $q$.

Recall that for a positive integer $\delta$ and a residue class $g\pmod {\delta}$,
the  {generalized Dedekind eta-function} $\eta_{\delta, g}(\tau)$ is defined by
\begin{align}\label{gen}
\eta_{\delta,g}(\tau)=q^{\frac{\delta}{2}P_2\left(\frac{g}{\delta}\right)}
\prod_{{n>0\atop n \equiv g(\text{mod} {\delta})}}(1-q^n)
\prod_{n>0\atop n \equiv -g(\text{mod} {\delta})}(1-q^n),
\end{align}
where
\[P_2(t)=\{t\}^2-\{t\}+\frac{1}{6}\]
is the second Bernoulli function and $\{t\}$ is the fractional part of $t$,
see, for example, \cite{Robins-1994,Schoeneberg-1974}.
Note that
\begin{align}\label{ge-e}
  \eta_{\delta, 0}(\tau)=\eta^2(\delta\tau)\quad\text{and}\quad\eta_{\delta, \frac{\delta}{2}}(\tau)=\frac{\eta^2(\frac{\delta}{2}\tau)}{\eta^2(\delta\tau)}.
\end{align}
A generalized eta-quotient is a function of the form
\begin{align}\label{geq-def}
\prod_{\delta|M\atop 0\leq g<\delta}\eta_{\delta, g}^{r_{\delta, g}}(\tau),
\end{align}
where $M\geq 1$ and
\begin{align*}
r_{\delta, g}\in
\begin{cases}
\frac{1}{2}\mathbb{Z}, &\text{if}\; g=0\; \text{or}\; g=\frac{\delta}{2},\\[3pt]
\mathbb{Z}, &\text{otherwise},
\end{cases}
\end{align*}
see, for example, Robins \cite{Robins-1994}.
In view of \eqref{ge-e}, when $g=0$ or $g=\frac{\delta}{2}$,
if $r_{\delta, g}\in\frac{1}{2}\mathbb{Z}$,
then the powers of the  eta-functions in \eqref{geq-def} are  integers.

For partition functions $b(n)$ as defined  in \eqref{def},
our algorithm can be extended to
derive Ramanujan-type identities on $b(mn+t)$ for $m>0$ and $0\leq t\leq m-1$,
such as the Ramanujan-type identities on Andrews' singular overpartition functions $\overline{Q}_{3,1}(9n+3)$ and $\overline{Q}_{3,1}(9n+6)$ due to  Shen \cite{Shen-2016}. The extended algorithm can also be employed to
derive dissection formulas,
such as the $2$-dissection formulas of Ramanujan, first proved by Andrews \cite{Andrew-1981}, and the $8$-dissection formulas due to Hirschhorn \cite{Hirschhorn-2001}.

\section{Finding Modular Functions for $\Gamma_1(N)$}\label{construction}

For the partition functions $a(n)$ as defined by \eqref{radu-gf}, namely,
\begin{align*}
  \sum_{n=0}^{\infty}
a(n)q^n=\prod_{\delta | M}
(q^\delta;q^\delta)^{r_\delta}_\infty,
\end{align*}
where $M$ is a positive integer and
$r_\delta$ are integers,
Radu \cite{Radu-2009} defined
\begin{align}\label{gmt}
g_{m,t}(\tau)  = q^{\frac{t-\ell}{m}}\sum\limits_{n=0}^\infty a(mn+t)q^n,
\end{align}
where
\[\ell=-\frac{1}{24}\sum\limits_{\delta|M}\delta r_\delta.\]

Let $\phi(\tau)$ be a generalized eta-quotient,
and let $F(\tau) = \phi(\tau) g_{m, t}(\tau)$.
The objective of this section is to give a criterion for $F(\tau)$ to be a modular function for $\Gamma_1(N)$.
We find that the transformation formula
for $g_{m,t}(\tau)$ under $\Gamma_1(N)^*$
is analogous to the transformation formula of
Radu \cite[Lemma~2.14]{Radu-2009} with respect to $\Gamma_0(N)^*$.
Then we utilize the transformation laws of
Newman \cite{Newman-1959} and Robins \cite{Robins-1994} to obtain the
transformation formula of $F(\tau)$.
With the aid of the Laurent expansions of $\phi(\tau)$ and $g_{m, t}(\tau)$,
we obtain a necessary and sufficient condition for $F(\tau)$
to be a modular function for $\Gamma_1(N)$.

We first state the conditions on $N$.
In fact, we make the following changes on the conditions on $N$ given by Definition 34 and Definition 35 in \cite{Radu-2015}:
Change the condition $\delta| mN$
for every $\delta|M$ with $r_\delta\neq 0$
to $M|N$,
and add the following condition \nameref{con_N_10}.
For completeness,
we list all the conditions on $N$.
Let $\kappa=\gcd(m^2-1, 24)$.
Assume that $N$ satisfies the following conditions:
\begin{description}\label{con}
\setlength{\parskip}{2ex}
  \item[{1}\label{con_N_1}] $M|N$.

  \item[{2}\label{con_N_2}] $p| N$ for any prime $p| m$.

  \item[{3}\label{con_N_6}] $\kappa N\sum\limits_{\delta|M}r_\delta \equiv 0 \pmod{8}$.

  \item[{4}\label{con_N_7}] $\kappa mN^2\sum\limits_{\delta|M}\frac{r_\delta}{\delta}\equiv 0 \pmod{24}$.

  \item[{5}\label{con_N_8}] $\frac{24m}{\gcd(\kappa \alpha(t), 24m)}\left| N\right.$,
  where $\alpha(t) = -\sum_{\delta|M}\delta r_{\delta}-24t$.

  \item[{6}\label{con_N_9}] Let $\prod_{\delta|M}\delta^{|r_\delta|} = 2^zj$,
   where $z\in \mathbb{N}$ and $j$ is odd.
   If $2| m$, then $\kappa N\equiv 0 \pmod 4$ and $Nz\equiv 0 \pmod 8$,
    or $z\equiv 0 \pmod 2$ and $N(j-1)\equiv 0 \pmod 8$.

  \item[{7}\label{con_N_10}]
  Let $\mathbb{S}_n = \{j^2 \pmod n \colon j\in\mathbb{Z}_n,\  \gcd(j, n)=1,\  j\equiv 1 \pmod{N}\}$.
  For any $s\in \mathbb{S}_{24mM}$,
 \[\frac{s-1}{24}\sum\limits_{\delta|M}\delta r_\delta + ts \equiv t \pmod m.\]
\end{description}
Note that there always exists $N$ satisfying the above
conditions, because $N=24mM$ would make a feasible choice.
From now on,
we denote by $\gamma$ the matrix $\left({a\,\,b\atop c\,\,d}\right)$.

\begin{thm}\label{con_F_modular_function}
For a given partition function $a(n)$ as defined by \eqref{radu-gf},
and for given integers $m$ and $t$,
suppose that $N$ is a positive integer satisfying the conditions \nameref{con_N_1}--\nameref{con_N_10}.
Let
\begin{align*}
  F(\tau)=\phi(\tau)\, g_{m,t}(\tau),
\end{align*}
where
\begin{align}\label{phi-form}
\phi(\tau)=\prod_{\delta | N}\eta^{a_{\delta}}(\delta \tau)\,\prod_{{\delta|N\atop 0<g\leq \left\lfloor\delta/2\right\rfloor}}\eta_{\delta,g}^{a_{\delta,g}}(\tau),
\end{align}
and $a_{\delta}$ and $a_{\delta,g}$ are integers.
Then $F(\tau)$ is a modular function with respect to $\Gamma_1(N)$ if and only if $a_{\delta}$ and $a_{\delta,g}$ satisfy the following conditions:

\begin{enumerate}
\setlength{\parskip}{2ex}
  \item[{\rm(1)}\label{F_con_1}] $\sum\limits_{\delta|N}a_\delta+\sum\limits_{\delta|M}r_\delta=0$,

  \item[{\rm(2)}\label{F_con_2}] $N \sum\limits_{\delta|N}\frac{a_\delta}{\delta}
                  +2N\sum\limits_{{\delta|N\atop 0<g\leq \left\lfloor{\delta}/{2}\right\rfloor}}\frac{a_{\delta,g}}{\delta}
                  +Nm\sum\limits_{\delta|M}\frac{r_\delta}{\delta}
                  \equiv0\pmod{24}$,

  \item[{\rm(3)}\label{F_con_3}]
  $\sum\limits_{\delta|N}\delta a_\delta
  +12 \sum\limits_{{\delta|N\atop 0<g\leq \left\lfloor{\delta}/{2}\right\rfloor}}\delta P_2\left(\frac{g}{\delta}\right){a_{\delta,g}}
  +m\sum\limits_{\delta|M}\delta r_\delta
  +\frac{(m^2-1)\alpha(t)}{m }
  \equiv0\pmod{24},$

  \item[{\rm(4)}\label{F_con_4}] For any integer $0<a<12N$ with $\gcd{(a,6)}=1$ and $a\equiv 1\pmod N$,
  $$\prod\limits_{\delta|N}\left(\frac{\delta}{a}\right)^{|a_\delta|}
  \prod\limits_{\delta|M}
  \left(\frac{m\delta}{a}\right)^{|r_\delta|} e^{\sum\limits_{\delta|N}\sum\limits_{g=1}^{{\tiny \left\lfloor\delta/2\right\rfloor}}\pi i\big(\frac{g}{\delta}-\frac{1}{2}\big)(a-1)a_{\delta,g}}=1.$$
\end{enumerate}
\end{thm}

For example,  consider the overpartition function $\overline{p}(n)$.
Recall that an overpartition of a positive integer $n$ is a partition of $n$ where the first occurrence of each distinct part may be overlined, and the number of overpartitions
of $n$ is denoted by $\overline{p}(n)$ for $n\geq 1$ and $\overline{p}(0)=1$.
As noted by Corteel and Lovejoy \cite{Corteel-Lovejoy},
the generating function of $\overline{p}(n)$ is given by
\begin{align*}
\sum_{n=0}^{\infty}\overline{p}(n)q^n=\frac{(q^2;q^2)_{\infty}}{(q;q)^2_{\infty}}.
\end{align*}
For the overpartition function $\overline{p}(5n+2)$, we see that
$N=10$ satisfies the conditions \nameref{con_N_1}--\nameref{con_N_10}.
Next we proceed to find a generalized eta-quotient $\phi(\tau)$ such that $\phi(\tau) g_{5,2}(\tau)$ is a modular function for $\Gamma_1(10)$.
By the above theorem, the function
\[\prod_{\delta|10}\eta^{a_\delta}(\delta\tau)\,
\prod_{{\delta|10\atop 0<g\leq \left\lfloor\delta/2\right\rfloor}}\eta_{\delta,g}^{a_{\delta,g}}(\tau)\, g_{5,2}(\tau)\]
is a modular function for $\Gamma_1(10)$
if and only if $a_\delta$ and $a_{\delta,g}$ fulfill the following conditions:
\begin{align}\label{overpartition_ex}
\begin{cases}
a_1+a_2+a_5+a_{10}-1=0,\\[6pt]
10a_1 +5  a_2+10  a_{2,1}+2  a_5 +4  a_{5,1}+4  a_{5,2}+a_{10}+2  a_{10,1}\\[6pt]
\quad\quad\quad\quad\quad\quad+2  a_{10,2}+2  a_{10,3}+2  a_{10,4}+2   a_{10,5}-3\equiv 0 \pmod{24},\\[6pt]
a_1+2  a_2-2  a_{2,1}+5  a_5+\frac{2  a_{5,1}}{5}-\frac{22  a_{5,2}}{5}+10   a_{10}+\frac{46  a_{10,1}}{5}\\[6pt]
\quad\quad\quad\quad\quad\quad+\frac{4  a_{10,2}}{5}-\frac{26 a_{10,3}}{5}-\frac{44  a_{10,4}}{5}-10  a_{10,5}+\frac{48}{5}\equiv 0 \pmod{24},\\[6pt]
\left(\frac{10}{a}\right)
\prod\limits_{\delta|10}\left(\frac{\delta}{a}\right)^{|a_\delta|}
e^{\sum\limits_{\delta|10}\sum\limits_{g=1}^{{\tiny \left\lfloor\delta/2\right\rfloor}}\pi i\big(\frac{g}{\delta}-\frac{1}{2}\big)(a-1)a_{\delta,g}}=1,
\end{cases}
\end{align}
for any $0<a<120$ with $\gcd(a,6)=1$ and $a\equiv 1\pmod {10}$.
We find that
\begin{align*}
  (a_1,a_2,a_{2,1},a_5,a_{5,1},a_{5,2},a_{10},a_{10,1},&a_{10,2},a_{10,3},a_{10,4},a_{10,5})\\[6pt]
&=(0, 0, 0, 0, 0, 0,  1, 0, 0, 0, -8, 9)
\end{align*}
is an integer solution of \eqref{overpartition_ex}.
Let
\[\phi(\tau)=\frac{\eta(10\tau)\eta_{10,5}^9(\tau)}{\eta_{10,4}^8(\tau)}.\]
Since
\begin{align*}
g_{5,2}(\tau)=q^{\frac{2}{5}}
\sum_{n=0}^\infty \overline{p}(5n+2)q^n,
\end{align*}
we find that
\begin{align}\label{Ftau-Q31}
F(\tau)
=q^{\frac{2}{5}}
\phi(\tau)
  \sum_{n=0}^{\infty}\overline{p}(5n+2)q^n
\end{align}
is a modular function with respect to $\Gamma_1(10)$.

Let
\[\Gamma_1(N)^*=\left\{\left({a\,\,b\atop c\,\,d}
                           \right)\in \Gamma_1(N) \colon \gcd{(a,6)}=1,\  ac>0
\right\}.\]
The following lemma asserts that
the invariance of the function $f(\tau)$ under $\Gamma_1(N)$
is equivalent to the invariance under $\Gamma_1(N)^*$.

\begin{lem}\label{gamma1-transformation}
Let $k$ be an integer, $N$ be a positive integer and
$f\colon\mathbb{H} \rightarrow \mathbb{C}$ be a function such that
\begin{align}\label{lem-gamma1-tran-eq}
  f(\gamma\tau) = (c\tau + d)^{k}f(\tau)
\end{align}
for any $\gamma\in \Gamma_1(N)^*$. Then $f$ is weight-$k$ invariant under $\Gamma_1(N)$.
\end{lem}

\begin{proof}
Let
$$A=\left\{\left({a\,\,\,b\atop c\,\,\,d}\right)\in\Gamma_1(N)\colon\gcd(a, 6)=1\right\}.$$
By Lemma 3 of Newman \cite{Newman-1959},
we know that $\Gamma_1(N)$ is generated by $A$.
Hence it suffices to show that $$f(\gamma\tau) = (c\tau + d)^{k}f(\tau)$$
for any $\gamma\in A$.
By the condition of Lemma \ref{gamma1-transformation},
we may restrict our attention only to two cases.
(1) $\gamma\in A$, $a>0$ and $c\le 0$.
(2) $\gamma\in A$, $a<0$ and $c\ge 0$.
Here we only consider the first case,
and the second case can be justified in the same manner.
For the first case, since  $a>0$ and $c\le 0$,
 there exists a positive integer
$x$ such that $ax+\frac{c}{N}>0$.
Let
\begin{align*}
\gamma_1=\left({\,\,1\,\,\,\,\,\,0 \atop Nx\,\,\,1}\right)\quad \text{and}\quad \gamma_2=\left({\,\,a\,\,\quad\quad\quad\,b \atop Nax+c\,\,Nbx+d}\right).
\end{align*}
Then $\gamma_2=\gamma_1\gamma$ and $\gamma_1\in \Gamma_1(N)^*$.
Therefore,
\begin{align}\label{eq2_gms}
 f(\gamma_2\tau)  = f(\gamma_1(\gamma\tau)) =\left(Nx(\gamma\tau) +1\right)^{k}f(\gamma\tau).
\end{align}
Since $\gamma\in A$, we have $\gcd(a, 6)=1$,
and so $\gamma_2\in \Gamma_1(N)^*$.
Applying \eqref{lem-gamma1-tran-eq} with $\gamma_2\in \Gamma_1(N)^*$,
we get
\begin{align}\label{eq1_gms}
 f(\gamma_2\tau) = \left((Nax+c)\tau + (Nbx+d)\right)^{k}f(\tau).
\end{align}
Combining \eqref{eq2_gms} and \eqref{eq1_gms},
we deduce that
\[f(\gamma\tau) = \left(c\tau+d\right)^{k}f(\tau),\]
as claimed.
\end{proof}

The following transformation formula
for $g_{m,t}(\tau)$ under $\Gamma_1(N)^*$
is analogous to  the transformation formula of
Radu \cite[Lemma~2.14]{Radu-2009} with respect to $\Gamma_0(N)^*$.
The proof parallels that of Radu, and hence it is omitted.

\begin{lem}\label{tran_g_gamma-1}
For a given partition function $a(n)$ as defined by \eqref{radu-gf},
and for given integers $m$ and $t$,
let $N$ be a positive integer satisfying the above conditions \nameref{con_N_1}--\nameref{con_N_10}.
For any $\gamma \in \Gamma_1(N)^*$,
we have
\begin{align*}
g_{m,t}(\gamma\tau) = (c\tau+d)^{\frac{1}{2}\sum\limits_{\delta|M}r_\delta}\, e^{\pi i\zeta(\gamma)}\,\prod\limits_{\delta|M}L(m\delta c,a)^{|r_\delta|}
\, g_{m, t}(\tau),
\end{align*}
where
\begin{align*}
L(c,a)=
\begin{cases}
\left(\frac{c}{a} \right), & \hbox{if $a>0$}, \\[6pt]
\left(\frac{-c}{-a} \right), & \hbox{otherwise,}
\end{cases}
\end{align*}
$\left(\frac{\ }{\ } \right)$ is the
Jacobi symbol,
\begin{align*}
\zeta(\gamma)&=\frac{ab(m^2-1)\alpha(t)}{12m}
+\frac{abm}{12}\sum\limits_{\delta|M}\delta r_{\delta}-
\frac{acm}{12}\sum\limits_{\delta|M}\frac{r_\delta}{\delta}
+\frac{\mathrm{sgn}(c)\,(a-1)}{4}\sum\limits_{\delta|M}r_\delta,
\end{align*}
and $\alpha(t)$ is defined as in the condition \nameref{con_N_8}.
\end{lem}

Next we derive a transformation formula for $F(\tau)$ under $\Gamma_1(N)^*$.
Recall the notation of Schoeneberg  \cite{Schoeneberg-1974}:
\begin{align}\label{Schoeneberg's notation}
\eta_{g,h}^{(s)}(\tau)= \alpha_0(h){e}^{\pi {i}P_2(\frac{g}{\delta})\tau}\rule{-10pt}{0pt}\prod\limits_{{m>0\atop m\equiv g(\text{mod}{\delta})}}\rule{-10pt}{0pt}
(1-\zeta_\delta^h{e}^{\frac{2\pi{i}\tau}{\delta}m})
\rule{-10pt}{0pt}\prod\limits_{{m>0 \atop m\equiv -g(\text{mod}{\delta})}}\rule{-10pt}{0pt}(1-\zeta_{\delta}^{-h}{e}^{\frac{2\pi{i}\tau}{\delta}m}),
\end{align}
where $\zeta_\delta$ is a primitive $\delta$-th root of unity,
\begin{eqnarray*}
\alpha_0(h) =
\begin{cases}
(1-\zeta_\delta^{-h}){e}^{\pi{i}P_1\left(\frac{h}{\delta}\right)}, & \text{if $g\equiv 0$  (mod $\delta$) and $h \not\equiv 0$
(mod $\delta$)},\\[3pt]
1, & \text{otherwise},
\end{cases}
\end{eqnarray*}
the first Bernoulli function $P_1(x)$ is given by
 \begin{eqnarray*}
P_1(x) =
\begin{cases}
x-\lfloor x\rfloor-\frac{1}{2}, & \text{if}\,\,x\not\in \mathbb{Z}, \\[3pt]
 0, & \text{otherwise},
\end{cases}
\end{eqnarray*}
and $\lfloor x\rfloor$ is the greatest integer less than or equal to $x$.
Since
\begin{align*}
\eta_{\delta,g}(\tau)=q^{\frac{\delta}{2}P_2\left(\frac{g}{\delta}\right)}
\prod_{{n>0\atop n \equiv g(\text{mod}{\delta})}}(1-q^n)
\prod_{n>0\atop n \equiv -g(\text{mod}{\delta})}(1-q^n),
\end{align*}
we have
\begin{align}\label{S_s}
  \eta_{\delta, g}(\tau) = \eta_{g, 0}^{(s)}(\delta\tau).
\end{align}

\begin{lem}\label{trans-F-lem}
For a given partition function $a(n)$ as defined by \eqref{radu-gf},
and for given integers $m$ and $t$,
let $N$ be a positive integer satisfying the conditions \nameref{con_N_1}--\nameref{con_N_10}, and
\begin{align}\label{thm-con_F}
  F(\tau)=\prod_{\delta | N}\eta^{a_{\delta}}(\delta \tau)\,\prod_{{\delta|N\atop 0<g\leq \left\lfloor\delta/2\right\rfloor}}\eta_{\delta,g}^{a_{\delta,g}}(\tau)\, g_{m,t}(\tau),
\end{align}
where $a_{\delta}$ and $a_{\delta,g}$ are integers.
Then for any $\gamma\in\Gamma_1(N)^*$,
\begin{align}\label{F-trans-law-1}
F(\gamma\tau)=
\prod_{\delta |N}L\left(\frac{c}{\delta},a\right) ^{a_\delta}
\prod\limits_{\delta|M}L&(m\delta c,a)^{|r_\delta|}
e^{\pi i \left(\nu(\gamma)+\xi(\gamma)\right)}\notag\\[2pt]
&\times(c\tau+d)^{\frac{1}{2}\Big(\sum\limits_{\delta|N}a_\delta+
\sum\limits_{\delta|M}r_\delta\Big)}
F{(\tau)},
\end{align}
where
\begin{align}\label{nu-def}
\nu(\gamma)=\sum\limits_{\delta|N\atop 0<g\leq \lfloor\delta/2\rfloor}\left(\frac{g}{\delta}-\frac{1}{2}\right)(a-1) a_{\delta,g}
\end{align}
and
\begin{align}
\xi(\gamma)&=\frac{a-1}{4}\sgn{(c)}\bigg(\sum\limits_{\delta|N}a_\delta+\sum\limits_{\delta|M}r_\delta\bigg)\notag\\[6pt]
&\quad -ac\bigg(
\sum\limits_{\delta|N}\frac{a_\delta}{12\delta}
                  +\sum\limits_{{\delta|N\atop 0<g\leq \left\lfloor{\delta}/{2}\right\rfloor}}\frac{a_{\delta,g}}{6\delta}
                  +\sum\limits_{\delta|M}\frac{mr_\delta}{12\delta}
\bigg)\notag\\[6pt]
&\quad +ab \bigg(
\sum\limits_{\delta|N}\frac{\delta a_\delta}{12}
+\sum\limits_{{\delta|N\atop 0<g\leq \left\lfloor{\delta}/{2}\right\rfloor}}\delta P_2\left(\frac{g}{\delta}\right){a_{\delta,g}}
+\sum\limits_{\delta|M}\frac{m\delta r_\delta}{12}
+\frac{(m^2-1)\alpha(t)}{12m}
\bigg).\label{xi-def}
\end{align}
\end{lem}

\begin{proof}
For any $\gamma=\left({a\,\,b\atop c\,\,d}\right)\in\Gamma_1(N)^*$,
we have
\begin{align}\label{F_trans_for-1}
 F(\gamma\tau)=\prod_{\delta | N}\eta^{a_{\delta}}(\delta \gamma\tau) \prod_{{\delta|N\atop 0<g\leq \left\lfloor\delta/2\right\rfloor}}\eta_{\delta,g}^{a_{\delta,g}}(\gamma\tau)\, g_{m,t}(\gamma\tau).
\end{align}
For any $\delta| N$,
let $\gamma_\delta'=\left({a\,\,\delta b\atop \frac{c}{\delta}\,\,\,d\,}\right)$.
Since $\gamma\in\Gamma_1(N)^*$, we have $N|c$ and so $\delta|c$ for any $\delta|N$. It follows that $\gamma_\delta'\in \Gamma$.
Thus \eqref{F_trans_for-1} can be written as
\begin{align}\label{expression_F-1}
 F(\gamma\tau)=\prod_{\delta | N}\eta^{a_{\delta}}(\gamma_\delta'(\delta\tau))\prod_{{\delta|N\atop 0<g\leq \left\lfloor\delta/2\right\rfloor}}\eta_{\delta,g}^{a_{\delta,g}}(\gamma\tau)\, g_{m,t}(\gamma\tau).
\end{align}
The transformation formula of Newman \cite[Lemma 2]{Newman-1959} states that
for any $\gamma\in\Gamma$ with $a>0$, $c>0$ and $\gcd(a,6)=1$,
\begin{align*}
  \eta(\gamma\tau)=\left(\frac{c}{a}\right)e^{-\frac{a\pi i}{12}(c-b-3)}(-i(c\tau+d))^{\frac{1}{2}}\eta(\tau).
\end{align*}
Therefore, for any $\gamma\in\Gamma$ with $ac>0$ and $\gcd(a,6)=1$,
we have
\begin{align}\label{trans-eta-newman}
\eta(\gamma\tau) = L(c,a)\, e^{\pi i(\frac{a}{12}(-c+b)+\frac{a-1}{4}
  \mathrm{sgn}(c))}(c\tau+d)^{\frac{1}{2}}\eta(\tau).
\end{align}
Since $\gamma\in\Gamma_1(N)^*$,
we see that
$\gcd(a, 6)=1$ and $ac>0$.
Applying
the transformation formula \eqref{trans-eta-newman} to each $\gamma_\delta'$,
we deduce that
\begin{align}\label{prod_eta_trans-1}
\prod_{\delta |N}\eta^{a_{\delta}}(\gamma_\delta'(\delta\tau))
=\prod_{\delta |N} L\left(\frac{c}{\delta},a\right) ^{a_\delta}  e^{\pi i\left(\frac{a}{12}(-\frac{c}{\delta}+\delta b)+ \frac{a-1}{4}
  \mathrm{sgn}(c)\right)a_\delta} (c\tau+d)^{\frac{a_\delta}{2}}\eta^{a_\delta}(\delta\tau).
\end{align}
Using the transformation formula of Robins \cite[Theorem~2]{Robins-1994}:
\begin{align*}
\eta_{\delta, g}(\gamma\tau) = e^{\pi i(\delta abP_2\left(\frac{g}{\delta}\right) - \frac{ac}{6\delta} + (a-1)(\frac{g}{\delta}- \frac{1}{2}))} \eta_{\delta, g}(\tau),
\end{align*}
we find that
\begin{align}\label{eta_trans_F-1}
\prod_{{\delta|N\atop 0<g\leq \left\lfloor\delta/2\right\rfloor}}\eta_{\delta,g}^{a_{\delta,g}}(\gamma\tau)
=
\prod_{{\delta|N\atop 0<g\leq \left\lfloor\delta/2\right\rfloor}}
e^{\pi i(\delta abP_2\left(\frac{g}{\delta}\right) - \frac{ac}{6\delta} + (a-1)(\frac{g}{\delta}- \frac{1}{2}))a_{\delta,g}} \eta_{\delta,g}^{a_{\delta,g}}(\tau).
\end{align}
Substituting the transformation formulas in \eqref{prod_eta_trans-1},
\eqref{eta_trans_F-1} and Lemma \ref{tran_g_gamma-1} into
\eqref{expression_F-1}, we reach the transformation formula \eqref{F-trans-law-1}.
\end{proof}

To prove Theorem \ref{con_F_modular_function},
we need the Laurent expansions of $g_{m,t}(\gamma\tau)$ and $\phi(\gamma\tau)$ for $\gamma\in\Gamma$.
Let us recall the two maps $p\colon \Gamma \times \mathbb{Z}_m \rightarrow \mathbb{Q}$
and $p\colon \Gamma\rightarrow \mathbb{Q}$
defined by Radu \cite{Radu-2009}, namely, for $\gamma\in\Gamma$ and $\lambda\in \mathbb{Z}_m$,
\begin{align}\label{def_p_gamma-l-1}
p(\gamma, \lambda)  =  \frac{1}{24}\sum\limits_{\delta|M}\frac{\gcd^2(\delta(a+\kappa\lambda c), mc)}{\delta m}r_\delta
\end{align}
and for $\gamma\in\Gamma$,
\begin{align}\label{def_p_gamma-1}
p(\gamma) = \min\{p(\gamma, \lambda)\colon \lambda=0, 1, \ldots, m-1\}.
\end{align}
The parabolic subgroup of $\Gamma$ is defined by
\[\Gamma_\infty=\left\{\pm\left({1\,\,b\atop 0\,\,1}\right)\colon b\in\mathbb{Z}\right\}.\]
For any $\gamma\in\Gamma$,
the $(\Gamma_1(N), \Gamma_\infty)$-double coset of $\gamma$ is given by
\[\Gamma_1(N)\gamma\Gamma_\infty = \{\gamma_{N} \gamma \gamma_\infty\colon
\gamma_{N}\in\Gamma_1(N),\, \gamma_\infty\in\Gamma_\infty\}.\]
Assume that $\Gamma$ has the following disjoint decomposition
\begin{align}\label{double_decomp}
\Gamma = \bigcup_{i = 1}^\epsilon \Gamma_1(N)\gamma_i\Gamma_\infty,
\end{align}
where $R = \{\gamma_1, \gamma_2, \ldots, \gamma_\epsilon\} \subseteq \Gamma$.
Denote the set of $(\Gamma_1(N), \Gamma_\infty)$-double cosets in $\Gamma$ by $\Gamma_1(N)\backslash \Gamma/\Gamma_\infty$.
Then \eqref{double_decomp} can be written as
\[\Gamma_1(N)\backslash \Gamma/\Gamma_\infty = \{\Gamma_1(N)\gamma\Gamma_\infty\colon \gamma\in R\}.\]
We say that $R$ is a complete set of representatives
of the double cosets $\Gamma_1(N)\backslash \Gamma/\Gamma_\infty$.

The following lemma gives a Laurent expansion of $g_{m,t}(\gamma\tau)$,
and the proof is similar
to that of Lemma 3.4 in Radu \cite{Radu-2009}, and hence it is omitted.

\begin{lem}\label{thm-definition-pgamma-1}
For a given partition function $a(n)$ as defined by \eqref{radu-gf},
and for given integers $m$ and $t$,
let $N$ be a positive integer satisfying the conditions \nameref{con_N_1}--\nameref{con_N_10},
and $R = \{\gamma_1, \gamma_2, \ldots, \gamma_\epsilon\}$
be a complete set of representatives
of the double cosets $\Gamma_1(N)\backslash \Gamma/\Gamma_\infty$.
For any $\gamma \in \Gamma$,
assume that $\gamma\in\Gamma_{1}(N)\gamma_i\Gamma_\infty$ for
some $1\leq i\leq \epsilon$.
Then there exists an integer $w$ and a Taylor series $h(q)$ in powers of $q^{\frac{1}{w}}$,
such that
$$g_{m,t}(\gamma\tau) = (c\tau+d)^{\frac{1}{2}\sum\limits_{\delta|M}r_\delta}q^{p(\gamma_i)}h(q).$$
\end{lem}

The following lemma gives a Laurent expansion
of $\phi(\gamma\tau)$ for any $\gamma\in\Gamma$.

\begin{lem}\label{phi_trans}
Let
\begin{align*}
\phi(\tau)= \prod_{\delta | N}\eta^{a_{\delta}}(\delta \tau)\,\prod_{{\delta|N\atop 0<g\leq \left\lfloor\delta/2\right\rfloor}}\eta_{\delta,g}^{a_{\delta,g}}(\tau),
\end{align*}
where $a_{\delta}$ and $a_{\delta,g}$ are integers.
For any $\gamma\in\Gamma$, there exists a positive integer $w$ and a Taylor series $h^*(q)$ in powers of $q^{\frac{1}{w}}$
such that
\begin{align*}
\phi(\gamma\tau)=(c\tau+d)^{\frac{1}{2}\sum\limits_{\delta|N}a_{\delta}}q^{p^*(\gamma)}h^*(q),
\end{align*}
where
\begin{align*}
p^*(\gamma)=\frac{1}{24}\sum_{\delta|N}\frac{\gcd^2{(\delta,c)}}{\delta}a_\delta
+\frac{1}{2}\sum_{{\delta|N\atop 0<g\leq \left\lfloor\delta/2\right\rfloor}}\frac{\gcd^2{(\delta,c)}}{\delta}P_2\left(\frac{a g}{\gcd{(\delta,c)}}\right) a_{\delta,g}.
\end{align*}
Furthermore, for any $\gamma_1\in\Gamma$ and $\gamma_2\in\Gamma_1(N)\gamma_1\Gamma_\infty$,
we have $p^*(\gamma_1)=p^*(\gamma_2)$.
\end{lem}

\begin{proof}
For any $\gamma=\left(a\,\,b\atop c\,\,d\right)\in \Gamma$,
we have
\begin{align*}
\phi(\gamma \tau)=\prod_{\delta | N}\eta^{a_{\delta}}(\delta \gamma\tau)\prod_{{\delta|N\atop 0<g\leq \left\lfloor\delta/2\right\rfloor}}\eta_{\delta,g}^{a_{\delta,g}}(\gamma\tau).
\end{align*}
It follows from \eqref{S_s} that
\begin{align}\label{lem-eq1}
\phi(\gamma \tau)=\prod_{\delta | N}\eta^{a_{\delta}}(\delta \gamma\tau)\prod_{{\delta|N\atop 0<g\leq \left\lfloor\delta/2\right\rfloor}}{\eta_{g,0}^{(s)}}^{a_{\delta,g}}(\delta\gamma\tau).
\end{align}
Since $\gcd(a,c)=1$, for any $\delta|N$,
there exist integers $x_\delta$ and $y_\delta$ such that \[ \delta a x_\delta + c y_\delta=\gcd(\delta a,c)=\gcd(\delta,c),\]
and hence
\begin{align}\label{mat_dec}
\left(\delta a\,\,\,\delta b\atop c\,\,\,\,\,\,\,d\right)=\left(\frac{\delta a}{\gcd(\delta,c)}\,\,-y_\delta\atop \frac{c}{\gcd(\delta,c)}\,\,\,\,\,\,x_\delta\,\,\,\right)
\left(\gcd(\delta,c)\,\,\,\,\delta b x_\delta + dy_\delta
\atop \,\,\,\,\,0\,\,\,\,\,\,\,\,\,\,\,\,\,\,\,\,\,\,\,\frac{\delta}{\gcd(\delta,c)}\right).
\end{align}
Set
\[\gamma_\delta=\left(\frac{\delta a}{\gcd(\delta,c)}\,\,\,\,-y_\delta\atop \frac{c}{\gcd(\delta,c)}\,\,\,\,\,\,\,\,x_\delta\,\,\right)
\qquad\text{and}\qquad
T_\delta= \left(\gcd(\delta,c)\,\,\,\,\delta b x_\delta + dy_\delta
\atop \,\,\,\,\,0\,\,\,\,\,\,\,\,\,\,\,\,\,\,\,\,\,\,\,\frac{\delta}{\gcd(\delta,c)}\right).\]
Note that $\gamma_\delta\in\Gamma$.
Combining \eqref{lem-eq1} and \eqref{mat_dec}, we deduce that
\begin{align}\label{F_tran}
\phi(\gamma \tau)=
\prod_{\delta | N}\eta^{a_{\delta}}(\gamma_\delta T_\delta \tau)\prod_{{\delta|N\atop 0<g\leq \left\lfloor\delta/2\right\rfloor}}{\eta_{g,0}^{(s)}}^{a_{\delta,g}}(\gamma_\delta T_\delta \tau).
\end{align}
By the transformation law for $\eta(\tau)$ under $\Gamma$ \cite[p.~145]{Rademacher-1973}, namely,
there exists a map $\varepsilon'\colon \Gamma\rightarrow\mathbb{C}$ such that
for any $\gamma\in\Gamma$,
\begin{align*}
\eta(\gamma \tau)
= \varepsilon'(\gamma)(c\tau+d)^{\frac{1}{2}} \eta(\tau),
\end{align*}
and the transformation formula for $\eta_{g,h}^{(s)}(\tau)$ under $\Gamma$
in \cite[p.~199~(30)]{Schoeneberg-1974}, namely, when $0<g<\delta$,
there exists a map $\epsilon_1\colon \Gamma\rightarrow \mathbb{C}$
such that for any $\gamma\in\Gamma$,
\begin{align*}
\eta_{g,h}^{(s)}(\gamma\tau) = \epsilon_1(\gamma)\, \eta_{g',h'}^{(s)}(\tau),
\end{align*}
where $g' = ag+ch$, $h'=bg+dh$,
it follows from  \eqref{F_tran} that there is a map $\varepsilon\colon \Gamma\rightarrow \mathbb{C}$ such that for any $\gamma\in\Gamma$,
\begin{align}\label{phi_trans_1}
\phi(\gamma \tau)=\varepsilon(\gamma)(c\tau+d)^{\frac{1}{2}\sum\limits_{\delta|N}a_\delta}
\prod_{\delta|N}\eta^{a_{\delta}}(T_\delta \tau)\prod_{{\delta|N\atop 0<g\leq \left\lfloor\delta/2\right\rfloor}}{\eta_{\frac{\delta a}{\gcd(\delta,c)}g,-y_\delta g}^{(s)}}^{a_{\delta,g}}(T_\delta \tau).
\end{align}
Substituting the $q$-expansions of the eta-function and the generalized eta-function into \eqref{phi_trans_1}, we see that there exists a positive integer $w$ and a Taylor series $h^*(q)$ in powers of $q^{\frac{1}{w}}$ such that
\begin{align*}
\phi(\gamma \tau)=(c\tau+d)^{\frac{1}{2}\sum\limits_{\delta|N}a_\delta}
q^{p^*(\gamma)}h^*(q).
\end{align*}

Next we aim to show that
$p^*(\gamma_1)=p^*(\gamma_2)$ for any $\gamma_1\in\Gamma$ and $\gamma_2\in\Gamma_1(N)\gamma_1\Gamma_\infty$.
Under the assumption that $\gamma_2\in\Gamma_1(N)\gamma_1\Gamma_\infty$,
there exist $\gamma_3\in\Gamma_1(N)$ and $\gamma_4\in \Gamma_\infty$
such that
\begin{align}\label{relation_gamma}
\gamma_2=\gamma_3\gamma_1\gamma_4.
\end{align}
Write
\[\gamma_1=\left(a_1\,\,b_1\atop c_1\,\,d_1\right),\quad\gamma_2=\left(a_2\,\,b_2\atop c_2\,\,d_2\right), \quad\gamma_3=\left(a_3\,\,b_3\atop c_3\,\,d_3\right), \quad\gamma_4=\left(\pm1\;\; b_4\atop \;\;0\;\;\pm1\right).\]
Owing to \eqref{relation_gamma}, we find that
\begin{align}\label{a_2_relation}
a_2=\pm (a_1a_3+b_3c_1)
\end{align}
and
\begin{align}\label{c_2_relation}
c_2=\pm (a_1c_3+c_1d_3).
\end{align}
For any $\delta|N$, since $\gamma_3\in\Gamma_1(N)$, we see that
$a_3\equiv 1\pmod{\delta}$,
$\delta|c_3$ and $\gcd(\delta, d_3) = 1$.
Using \eqref{a_2_relation}, it can be verified that
\begin{align}\label{a_2_equiv}
a_2g \equiv \pm a_1 g \pmod{\gcd{(\delta,c_1)}}.
\end{align}
In view of \eqref{c_2_relation}, we obtain that
\begin{align}\label{gcd_c_2}
\gcd(\delta,c_2)=\gcd(\delta,c_1).
\end{align}
Combining \eqref{a_2_equiv} and \eqref{gcd_c_2}, we arrive at
\begin{align}\label{P_2_equal}
P_2\left(\frac{a_1g}{\gcd(\delta,c_1)}\right)
=P_2\left(\frac{a_2g}{\gcd(\delta,c_2)}\right),
\end{align}
here we have used the fact that $P_2(-\alpha)=P_2(\alpha)$ for any $\alpha\in\mathbb{R}$.
Combining \eqref{gcd_c_2} and \eqref{P_2_equal}, we conclude that $p^*(\gamma_1)=p^*(\gamma_2)$, as claimed.
\end{proof}

We are now ready to complete the proof of Theorem \ref{con_F_modular_function}.

\begin{proof}[Proof of Theorem \ref{con_F_modular_function}]
Assume that
\begin{align}\label{F_def-1}
F(\tau)=\prod_{\delta | N}\eta^{a_{\delta}}(\delta \tau)\,\prod_{{\delta|N\atop 0<g\leq \left\lfloor\delta/2\right\rfloor}}\eta_{\delta,g}^{a_{\delta,g}}(\tau)\, g_{m,t}(\tau)
\end{align}
is a modular function with respect to $\Gamma_1(N)$,
where $a_{\delta}$ and $a_{\delta,g}$ are integers.
We proceed to show that the conditions (1)--(4) are fulfilled by the integers $a_{\delta}$ and $a_{\delta,g}$.

Since $\Gamma_1(N)^*\subseteq \Gamma_1(N)$ and $F(\tau)$ is a modular function for $\Gamma_1(N)$,
 for any $\gamma\in\Gamma_1(N)^*$, we have
\begin{align}\label{trans_F-1}
F(\gamma\tau)=F(\tau).
\end{align}
To compute $F(\gamma\tau)$, we need
the   transformation formula for $F(\tau)$ under $\Gamma_1(N)^*$
as given in Lemma \ref{trans-F-lem}, that is, for any $\gamma\in\Gamma_1(N)^*$,
\begin{align}\label{F-trans-law-2}
F(\gamma\tau)=
\prod_{\delta |N}L\left(\frac{c}{\delta},a\right) ^{a_\delta}
\prod\limits_{\delta|M}L&(m\delta c,a)^{|r_\delta|}
e^{\pi i \left(\nu(\gamma)+\xi(\gamma)\right)}\notag\\[2pt]
&\times(c\tau+d)^{\frac{1}{2}\big(\sum\limits_{\delta|N}a_\delta+
\sum\limits_{\delta|M}r_\delta\big)}
F{(\tau)},
\end{align}
where $\nu(\gamma)$ and $\xi(\gamma)$ are defined in \eqref{nu-def} and \eqref{xi-def}.
Combining \eqref{trans_F-1} and \eqref{F-trans-law-2},
we see that
\begin{align*}
\sum\limits_{\delta|N}a_\delta+\sum\limits_{\delta|M}r_\delta=0,
\end{align*}
thus (1) is satisfied.
Consequently, $\xi(\gamma)$ reduces to
\begin{align*}
&-ac\bigg(
\sum\limits_{\delta|N}\frac{a_\delta}{12\delta}
                  +\sum\limits_{{\delta|N\atop 0<g\leq \left\lfloor{\delta}/{2}\right\rfloor}}
                  \rule{-8pt}{0pt}\frac{a_{\delta,g}}{6\delta}
                  + \sum\limits_{\delta|M}\frac{mr_\delta}{12\delta}
\bigg)\\[6pt]
&\quad  +ab \bigg(
\sum\limits_{\delta|N}\frac{\delta a_\delta}{12}
+\sum\limits_{{\delta|N\atop 0<g\leq \left\lfloor{\delta}/{2}\right\rfloor}}
\rule{-8pt}{0pt}\delta P_2\left(\frac{g}{\delta}\right){a_{\delta,g}}
+\sum\limits_{\delta|M}\frac{m\delta r_\delta}{12}+\frac{(m^2-1)\alpha(t)}{12m}
\bigg).
\end{align*}
To prove (2), consider the matrix $\gamma=\left({1\,\,\,\, 0\atop N\,\,\,1\,}\right)\in\Gamma_1(N)^*$. In this case,
\eqref{F-trans-law-2} becomes
\begin{align}\label{F_trans_spec-1}
  F(\gamma\tau)=e^{-\pi iN\Big(
\sum\limits_{\delta|N}\frac{a_\delta}{12\delta}
                  +\sum\limits_{\delta|N}\sum\limits_{g=1}^{ \left\lfloor{\delta}/{2}\right\rfloor}\frac{a_{\delta,g}}{6\delta}
                  + \sum\limits_{\delta|M}\frac{mr_\delta}{12\delta}
\Big) }  F(\tau).
\end{align}
Hence (2) follows from \eqref{trans_F-1} and \eqref{F_trans_spec-1}.
Setting $\gamma=\left({1\quad\,1\,\,\,\atop N\,\,N+1}\right)\in\Gamma_1(N)^*$,  \eqref{F-trans-law-2} becomes
\begin{align*}
  F(\gamma\tau)
  =e^{\pi i \Big(
\sum\limits_{\delta|N}\frac{\delta a_\delta}{12}
+\sum\limits_{\delta|N}
\sum\limits_{g=1}^{\lfloor{\delta}/{2}\rfloor}
\delta P_2\left(\frac{g}{\delta}\right){a_{\delta,g}}
+\sum\limits_{\delta|M}\frac{m\delta r_\delta}{12}
+\frac{(m^2-1)\alpha(t)}{12m}
\Big)} F(\tau),
\end{align*}
which, together with \eqref{trans_F-1}, implies (3).
Using the conditions (1)--(3),
it can be checked that $\xi(\gamma)\equiv 0\pmod{2}$   for any $\gamma\in\Gamma_1(N)^*$.
It follows that
\[e^{\pi i\xi(\gamma)} = 1,\]
and so \eqref{F-trans-law-2} reduces to
\begin{align}\label{F-trans-law-2-1}
F(\gamma\tau)
&=\prod_{\delta |N}L\left(\frac{c}{\delta},a\right) ^{a_\delta}
\prod\limits_{\delta|M}L(m\delta c,a)^{|r_\delta|}
e^{\pi i \nu(\gamma)}(c\tau+d)^{\frac{1}{2}\big(\sum\limits_{\delta|N}a_\delta+
\sum\limits_{\delta|M}r_\delta\big)}
F{(\tau)}.
\end{align}
By the definition of $L$, we find that for any $\delta|N$,
\[L\Big(\frac{c}{\delta}, a\Big)=L(\delta c,a) = \left(\frac{\delta|c|}{|a|}\right),\]
and for any $\delta|M$,
\[L(m\delta c,a) = \left(\frac{m\delta|c|}{|a|}\right).\]
Hence \eqref{F-trans-law-2-1} is equivalent to
\begin{align}\label{F-trans-law-3-1}
F(\gamma\tau)
=\prod_{\delta |N}\left(\frac{\delta|c|}{|a|}\right)^{|a_\delta|}
\,
\prod\limits_{\delta|M}\left(\frac{m\delta|c|}{|a|}\right)^{|r_\delta|}
\,
e^{\pi i \nu(\gamma)}\, (c\tau+d)^{\frac{1}{2}\big(\sum\limits_{\delta|N}a_\delta+
\sum\limits_{\delta|M}r_\delta\big)}F(\tau).
\end{align}
In view of the condition (1), it is easily verified  that
\begin{align}\label{exp_tau_0-1}
(c\tau+d)^{\frac{1}{2}\big(\sum\limits_{\delta|N}a_\delta+
\sum\limits_{\delta|M}r_\delta\big)} = 1
\end{align}
and
\begin{align}\label{Jacobi_1-1}
\prod_{\delta |N}\left(\frac{|c|}{|a|}\right)^{|a_\delta|}
\,
\prod\limits_{\delta|M}\left(\frac{|c|}{|a|}\right)^{|r_\delta|} = 1.
\end{align}
Substituting \eqref{exp_tau_0-1} and \eqref{Jacobi_1-1} into \eqref{F-trans-law-3-1}
yields
\begin{align}\label{F-trans-law-4-1}
 F(\gamma\tau)=\prod_{\delta |N}\left(\frac{\delta}{|a|}\right)^{|a_\delta|}
\,
\prod\limits_{\delta|M}\left(\frac{m\delta}{|a|}\right)^{|r_\delta|}
\,
e^{\pi i \nu(\gamma)}\, F(\tau).
\end{align}
Comparing \eqref{trans_F-1} with \eqref{F-trans-law-4-1}, we deduce that
\begin{align}\label{thm-cond4-1}
\prod_{\delta |N}\left(\frac{\delta}{|a|}\right)^{|a_\delta|}
\,
\prod\limits_{\delta|M}\left(\frac{m\delta}{|a|}\right)^{|r_\delta|}
\,
e^{\pi i \nu(\gamma)}=1
\end{align}
for all integers $a$ with $\gcd{(a,6)}=1$ and $a\equiv 1\pmod N$.
Invoking the interpretation of the Jacobi symbol,
 we conclude that
\eqref{thm-cond4-1} holds for all integers $0<a<12 N$ with $\gcd{(a,6)}=1$ and $a\equiv 1\pmod N$.
This confirms (4).

Conversely, assume that the integers $a_\delta$, $a_{\delta, g}$
$(\delta|N, 0<g\le \lfloor\delta/2\rfloor)$ satisfy the
conditions
(1)--(4).
We proceed to show that
\begin{align*}
F(\tau)=\prod_{\delta | N}\eta^{a_{\delta}}(\delta \tau)\,\prod_{{\delta|N\atop 0<g\leq \left\lfloor\delta/2\right\rfloor}}\eta_{\delta,g}^{a_{\delta,g}}(\tau)\, g_{m,t}(\tau)
\end{align*}
is a modular function for $\Gamma_1(N)$.
It is clear that $F(\tau)$ is holomorphic on $\mathbb{H}$.

Based on the conditions (1)--(3),
it follows from Lemma \ref{trans-F-lem} that the transformation formula
\eqref{F-trans-law-4-1} for $F(\tau)$ holds for any $\gamma\in\Gamma_1(N)^*$.
Given the condition (4),
we see that
\eqref{thm-cond4-1} holds for all integers
$a$ with $\gcd{(a,6)}=1$ and $a\equiv 1\pmod N$. Combining \eqref{F-trans-law-4-1} and \eqref{thm-cond4-1}, we find that
for any $\gamma\in\Gamma_1(N)^*$,
\[F(\gamma\tau)=F(\tau).\]
In view of Lemma \ref{gamma1-transformation},
we conclude that $F(\gamma\tau)=F(\tau)$ for any $\gamma\in\Gamma_1(N)$.

It remains to show that for any $\gamma\in\Gamma$,
$F(\gamma\tau)$ has a Laurent expansion with a finite principal part
in powers of $q^{\frac{1}{N}}$.
Let $\gamma\in\Gamma$ and $R = \{\gamma_1, \gamma_2, \ldots, \gamma_\epsilon\}$ be a complete set of representatives of the double cosets $\Gamma_1(N)\backslash\Gamma/\Gamma_\infty$.
By the decomposition of $\Gamma$ in \eqref{double_decomp},
there exist an integer $1\leq i\leq \epsilon$ and matrices
$\gamma_N\in\Gamma_1(N)$, $\gamma_\infty\in\Gamma_\infty$ such that $\gamma=\gamma_N \gamma_i \gamma_\infty$.
By Lemma \ref{thm-definition-pgamma-1} and Lemma \ref{phi_trans},
there exist a positive integer $w$ and Taylor series $h(q)$ and $h^*(q)$ in powers of $q^{\frac{1}{w}}$ such that
\begin{align}\label{F_expansion-1}
F(\gamma\tau)=(c\tau+d)^{\frac{1}{2}\big(\sum\limits_{\delta|N}a_\delta
+\sum\limits_{\delta|M}r_\delta\big)}\, q^{p(\gamma_i)+p^*(\gamma_i)}\, h(q)\, h^*(q).
\end{align}
In view of the condition (1),   \eqref{F_expansion-1} reduces to
\begin{align}\label{F_Laurent-1}
  F(\gamma\tau)=q^{p(\gamma_i)+p^*(\gamma_i)}\, h(q)\, h^*(q),
\end{align}
which implies that there exists a positive integer $k$ such that $F(\gamma\tau)$ has the Laurent expansion with a finite principal part in powers of $q^{\frac{1}{k}}$.
Since we have shown that $F(\tau)$ is invariant under $\Gamma_1(N)$,
by Lemma 1.14 in \cite{Stein-Modular},
we obtain that for any $\gamma\in \Gamma$,
$F(\gamma\tau)$ is invariant under $\gamma^{-1}\Gamma_1(N)\gamma$.
Notice that
$\left({1\,\,N\atop 0\,\,\,1}\right)\in \gamma^{-1}\Gamma_1(N)\gamma$.
So $F(\gamma\tau)$ has period $N$, namely,
\[F(\gamma(\tau+N)) = F(\gamma\tau).\]
Thus $F(\gamma\tau)$ has a Laurent expansion in powers of $q^{\frac{1}{N}}$.
By \eqref{F_Laurent-1}, we see that this Laurent expansion
has at most finitely many negative terms.
So we reach the assertion that $F(\tau)$ is a modular function for $\Gamma_1(N)$.
\end{proof}

Given a generating function of $a(n)$ as defined in \eqref{radu-gf} and integers $m$ and $t$,  we can find an integer $N$ satisfying the conditions \nameref{con_N_1}--\nameref{con_N_10}.
If we are lucky, we may use Theorem \ref{con_F_modular_function} to find
integers $a_\delta$, $a_{\delta, g}$ $(\delta|N, 0<g\le \lfloor\delta/2\rfloor)$ satisfying the
conditions (1)--(4), which lead to a generalized eta-quotient
\[\phi(\tau)=\prod_{\delta | N}\eta^{a_{\delta}}(\delta \tau)\,\prod_{{\delta|N\atop 0<g\leq \left\lfloor\delta/2\right\rfloor}}\eta_{\delta,g}^{a_{\delta,g}}(\tau)\]
such  that
\begin{align}\label{F-modular-function}
  F(\tau)=\phi(\tau)\, g_{m,t}(\tau)
\end{align}
is a modular function.
It should be noted that such a modular function $F(\tau)$ may be not unique.
To derive  a Ramanujan-type identity for $a(mn+t)$,
we aim to express $F(\tau)$ as a linear combination of generalized eta-quotients over $\mathbb{Q}$.
To this end,
we first investigate the behavior of $F(\tau)$ at each cusp of $\Gamma_1(N)$.
Let us recall some terminology of modular functions, see, for example \cite{Diamond-Shurman-2005, Stein-Modular}.
For  $\gamma = \left({a\,\,b\atop c\,\,d}\right) \in \Gamma$,
the width $w_\gamma$ of $\frac{a}{c}$ relative to $\Gamma_1(N)$
is the minimal positive integer $h$ such that
\[\left({1\,\,h\atop 0\,\,1}\right)\in\gamma^{-1}\Gamma_1(N)\gamma.\]
Let $f(\tau)$ be a modular function for $\Gamma_1(N)$.
It is known that
$f(\gamma\tau)$ is invariant under $\gamma^{-1}\Gamma_1(N)\gamma$, see   \cite[Lemma 1.14]{Stein-Modular}.
So $f(\gamma\tau)$ has period $w_\gamma$, which implies that
$f(\gamma\tau)$ has a Laurent expansion in powers of $q^{1/w_\gamma}$.
Since $f(\tau)$ is a modular function,
this Laurent expansion has at most finitely many negative terms.
Write
\begin{align}\label{Lauren-expansion}
  f(\gamma\tau)=\sum\limits_{n=-\infty}^\infty b_nq^{n/w_\gamma},
\end{align}
where $b_n=0$ for almost all negative integers $n$.
Let $n_\gamma$ be the smallest integer such that $b_{n_\gamma}\neq 0$.
We call $n_\gamma$ the $\gamma$-order of $f$ at $\frac{a}{c}$, denoted by $\mathrm{ord}_\gamma(f)$.
Denote the smallest exponent of $q$
on the right hand side of \eqref{Lauren-expansion} by $v_\gamma$, so that
\begin{align}\label{v-eq}
\mathrm{ord}_\gamma(f) = v_\gamma w_\gamma.
\end{align}
Furthermore,
the order of $f$ at the cusp $ \frac{a}{c}\in \mathbb{Q}\cup \{\infty\}$ is defined by
\begin{equation}\label{ord_cusp_def}
\mathrm{ord}_{a/c}(f)=\mathrm{ord}_\gamma(f)
\end{equation}
for some $\gamma\in\Gamma$ such that $\gamma\infty=\frac{a}{c}$.
It is  known that $\mathrm{ord}_{a/c}(f)$ is well-defined, see \cite[p.~72]{Diamond-Shurman-2005}.

The following theorem gives
estimates of the orders of $F(\tau)$ at cusps of $\Gamma_1(N)$.

\begin{thm}\label{order_F-1}
For a given partition function $a(n)$ as defined by \eqref{radu-gf},
and for given integers $m$ and $t$,
let
\begin{align*}
  F(\tau)=\phi(\tau)\,g_{m,t}(\tau),
\end{align*}
where
\[\phi(\tau)=\prod_{\delta | N}\eta^{a_{\delta}}(\delta \tau)\prod_{{\delta|N\atop 0<g\leq \left\lfloor\delta/2\right\rfloor}}\eta_{\delta,g}^{a_{\delta,g}}(\tau),\] $a_{\delta}$ and $a_{\delta,g}$ are integers.
Assume that $F(\tau)$ is a modular function for $\Gamma_1(N)$.
Let $\{s_1, s_2,\ldots,s_\epsilon\}$ be a complete set of inequivalent cusps of $\Gamma_1(N)$,
and for each $1\leq i \leq \epsilon$,
let $\alpha_i\in\Gamma$ be such that $\alpha_i\infty = s_i$.
Then
\begin{align}\label{ord_cusp-1}
\mathrm{ord}_{s_i}(F(\tau))\ge w_{\alpha_i}\,(p(\alpha_i)+p^*(\alpha_i)),
\end{align}
where $p(\gamma)$ is given by \eqref{def_p_gamma-1} and $p^*(\gamma)$ is defined in Lemma \ref{phi_trans}.
\end{thm}

To compute the right hand side of \eqref{ord_cusp-1}, we need the following formula due to Cho, Koo and Park \cite{Cho-Koo-Park-2009}:
\begin{align}\label{w-calcu}
  w_\gamma=
\begin{cases}
1, &\text{if}\  N = 4\ \text{and}\ \gcd(c,4) =2, \\
\frac{N}{\gcd(c, N)}, & \text{otherwise},
  \end{cases}
\end{align}
where $\gamma = \left({a\,\,b\atop c\,\,d}\right) \in \Gamma$.
For example, consider  the modular function
\begin{align*}
F(\tau)
=q^{\frac{2}{5}}
\frac{\eta(10\tau)\eta_{10,5}^9(\tau)}{\eta_{10,4}^8(\tau)}
  \sum_{n=0}^{\infty}\overline{p}(5n+2)q^n
\end{align*}
for $\Gamma_1(10)$  as given in \eqref{Ftau-Q31}.
A complete set $\mathcal{S}(N)$ of inequivalent cusps of $\Gamma_1(N)$ has
been found in \cite[Corollary 4]{Cho-Koo-Park-2009}.
In particular, for $N=10$, we have
\begin{align}\label{S_10}
  \mathcal{S}(10)=\left\{0,\  \frac{1}{5},\ \frac{1}{4},\ \frac{3}{10}, \ \frac{1}{3},\  \frac{3}{5},\ \frac{1}{2},\  \infty
\right\}.
\end{align}
Employing Theorem \ref{order_F-1},
we obtain the following lower bounds of the orders of $F(\tau)$ at cusps of $\Gamma_1(10)$:
\begin{eqnarray*}
\begin{split}
  &\mathrm{ord}_0(F(\tau)) \geq -3,\  \mathrm{ord}_{1/5}(F(\tau)) \geq \frac{19}{5},\  \mathrm{ord}_{1/4}(F(\tau)) \geq -2,\ \\[6pt]
   &\mathrm{ord}_{3/10}(F(\tau)) \geq -\frac{18}{5},\
\mathrm{ord}_{1/3}(F(\tau)) \geq -3,\ \mathrm{ord}_{3/5}(F(\tau)) \geq \frac{27}{5},\ \\[6pt] &\mathrm{ord}_{1/2}(F(\tau)) \geq -2,\
\mathrm{ord}_{\infty}(F(\tau)) \geq -\frac{2}{5}.
\end{split}
\end{eqnarray*}
Notice that
$F(\tau)$ may have poles at some cusps not equivalent to infinity.

We are now ready to prove Theorem \ref{order_F-1}.

\begin{proof}[Proof of Theorem \ref{order_F-1}]
It is known that there exists a bijection from the set of all inequivalent cusps of
$\Gamma_1(N)$ to the double coset space $\Gamma_1(N)\backslash\Gamma/\Gamma_\infty$,
as given by
$$\Gamma_1(N)({a/c})\mapsto\Gamma_1(N)\left({a\,\,b\atop c\,\,d}\right)\Gamma_\infty,$$
see \cite[Proposition 3.8.5]{Diamond-Shurman-2005}.
Since $\{s_1, s_2,\ldots,s_\epsilon\}$ is a complete set of inequivalent cusps of $\Gamma_1(N)$ and $\alpha_i\infty = s_i$ for $1\leq i\leq \epsilon$,
 we see that $\{\alpha_1, \alpha_2, \ldots, \alpha_\epsilon\}$ is a complete set of representatives of $\Gamma_1(N)\backslash\Gamma/\Gamma_\infty$.
Applying  Lemma \ref{thm-definition-pgamma-1} with $\gamma_i = \alpha_i$,
we find that there exists an integer $w_1$ and a Taylor series $h(q)$ in powers of $q^{\frac{1}{w_1}}$ such that
\begin{align}\label{thm-ord-eq1}
g_{m,t}(\alpha_i\tau) = (c\tau+d)^{\frac{1}{2}\sum\limits_{\delta|M}r_\delta}q^{p(\alpha_i)}h(q).
\end{align}
By Lemma \ref{phi_trans},
there exists a positive integer $w_2$ and a Taylor series $h^*(q)$ in powers of $q^{\frac{1}{w_2}}$,
such that
\begin{align}\label{thm-ord-eq2}
\phi(\alpha_i\tau)=(c\tau+d)^{\frac{1}{2}\sum\limits_{\delta|N}a_{\delta}}q^{p^*(\alpha_i)}h^*(q).
\end{align}
Combining \eqref{thm-ord-eq1} and \eqref{thm-ord-eq2},
we get
\begin{align}\label{F_Laurent_2}
  F(\alpha_i\tau)=(c\tau+d)^{\frac{1}{2}\big(
  \sum\limits_{\delta|N}a_\delta+\sum\limits_{\delta|M}r_\delta\big)}q^{p(\alpha_i)+p^*(\alpha_i)}\, h(q)\, h^*(q).
\end{align}
Since $F(\tau)$ is a modular function for $\Gamma_1(N)$,
using the condition (1) in Theorem \ref{con_F_modular_function},
\eqref{F_Laurent_2} reduces to
\begin{align}\label{F_Laurent_3}
F(\alpha_i\tau)=q^{p(\alpha_i)+p^*(\alpha_i)}\, h(q)\, h^*(q).
\end{align}
Let $v_{\alpha_i}$ denote the smallest exponent of $q$
on the right hand side of \eqref{F_Laurent_3}.
The relation $\mathrm{ord}_\gamma(f) = v_\gamma w_\gamma$ as given in \eqref{v-eq} yields
\begin{align}\label{v-order}
  v_{\alpha_i} =\frac{{\rm ord}_{\alpha_i}(F(\tau))}{w_{\alpha_i}}.
\end{align}
Since $ h(q)$ and $h^*(q)$ are Taylor series,  it follows from \eqref{F_Laurent_3}
that
\begin{align}\label{ineq}
  v_{\alpha_i} \geq p(\alpha_i)+p^*(\alpha_i).
\end{align}
Combining \eqref{v-order} and \eqref{ineq},
we conclude that
\begin{align}\label{F_lower_gamma}
  \mathrm{ord}_{\alpha_i}(F(\tau))\ge w_{\alpha_i}\,(p(\alpha_i)+p^*(\alpha_i)).
\end{align}
By the definition \eqref{ord_cusp_def}, we have
\begin{align}\label{F-order-eq3}
  \mathrm{ord}_{s_i}(F(\tau)) = \mathrm{ord}_{\alpha_i}(F(\tau)).
\end{align}
Thus the estimate \eqref{ord_cusp-1} follows from \eqref{F_lower_gamma} and \eqref{F-order-eq3}.
\end{proof}

\section{Sketch of the Algorithm}\label{Sketch}

In this section, we give a sketch of our algorithm.
Given a generating function of $a(n)$ as defined
in \eqref{radu-gf} and integers $m$ and $t$,
we can find an integer $N$ satisfying
the conditions \nameref{con_N_1}--\nameref{con_N_10}.
Assume that we have found a generalized eta-quotient $\phi(\tau)$
such that
\begin{align}\label{F_expression}
F(\tau) = \phi(\tau)\,g_{m,t}(\tau)
\end{align}
is a modular function for $\Gamma_1(N)$.
To derive an expression of $F(\tau)$,
we consider a class of modular functions: the set of generalized eta-quotients which are modular functions for $\Gamma_1(N)$ with poles only at infinity, denoted by $GE^\infty(N)$.
Note that the notation
$E^\infty(N)$ is used by Radu \cite{Radu-2015} to denote
the set of modular eta-quotients with poles only at infinity for $\Gamma_0(N)$.
Our goal is to derive an expression of $F(\tau)$ in terms of the generators of $GE^\infty(N)$. Then we are led to a Ramanujan-type identity for $a(mn+t)$.

Our algorithm consists of the following steps:
\begin{description}
 \setlength{\parskip}{2ex}
  \item[Step 1] Use Theorem \ref{con_F_modular_function} to find a generalized eta-quotient $\phi(\tau)$ for which $F(\tau)$ in \eqref{F_expression} is a modular function for $\Gamma_1(N)$.

  \item[Step 2] Find a finite set $\{z_1,z_2,\ldots,z_k\}$ of generators of $GE^\infty(N)$ by utilizing a formula of Robins and the theory of Diophantine inequalities.

  \item[Step 3] Let $\langle GE^\infty(N)\rangle_{\mathbb{Q}}$ be the vector space over $\mathbb{Q}$ generated by generalized eta-quotients in $GE^\infty(N)$. Employ the Algorithm AB of Radu for $\Gamma_1(N)$ on
      $\{z_1,z_2,\ldots,z_k\}$ to generate a modular function $z$ and a $\mathbb{Q}[z]$-module basis $1, e_1, \ldots, e_w$ of
      $\langle GE^\infty(N)\rangle_{\mathbb{Q}}$.

  \item[Step 4] Find a generalized eta-quotient $h$ in terms of generators of $GE^\infty(N)$ for which the modular function $hF$ has a pole only at infinity.  Theorem \ref{order_F-1} can be used to compute the lower bounds of the orders of $hF$ at all cusps of $\Gamma_1(N)$.

  \item[Step 5] Determine whether $hF$ is in $\langle GE^\infty(N)\rangle_{\mathbb{Q}}$ by applying the Algorithm MW of Radu to $hF$, $z$ and $1, e_1, \ldots, e_w$.
      If this goal can be achieved,
      then $F$ can be expressed as a linear combination of generalized eta-quotients over $\mathbb{Q}$.
\end{description}

For example, let us consider the overpartition function $\overline{p}(5n+2)$.
In Sect. \ref{construction},
we found $N=10$ satisfies the conditions \nameref{con_N_1}--\nameref{con_N_10}.

\begin{description}
\setlength{\parskip}{2ex}
  \item[Step 1] As shown in \eqref{Ftau-Q31},
\begin{align*}
F(\tau)
=q^{\frac{2}{5}}
\frac{\eta(10\tau)\eta_{10,5}^9(\tau)}{\eta_{10,4}^8(\tau)}
  \sum_{n=0}^{\infty}\overline{p}(5n+2)q^n
\end{align*}
is a modular function with respect to $\Gamma_1(10)$.

\item[Step 2] We obtain the following generators of $GE^\infty(10)$:
\begin{eqnarray}\label{generator_10}
\begin{split}
& z=\frac{
\eta(\tau)\eta(5\tau)
}{\eta_{5,1}^2(\tau)\eta^2(10\tau)
\eta_{10,1}(\tau)},\quad
z_1=\frac{\eta^2(2\tau)\eta(5\tau)
\eta_{5,1}^2(\tau)}
{\eta(\tau)\eta^{2}(10\tau)
\eta_{10,1}^4(\tau)},\quad\\[6pt]
&
 z_2=\frac{\eta^3(5\tau)\eta_{5,1}^4(\tau)}
{\eta(\tau)\eta(2\tau)\eta(10\tau)
\eta_{10,1}^3(\tau)},\quad
z_3=\frac{
\eta(\tau)\eta_{5,1}^2(\tau)\eta^2(10\tau)}
{\eta^2(2\tau)\eta(5\tau)\eta_{10,1}^4(\tau)},\\[6pt]
&
z_4=\frac{\eta^4(\tau)
\eta_{5,1}^2(\tau)}
{\eta^3(2\tau)\eta(10\tau)
\eta_{10,1}^4(\tau)}.
\end{split}
\end{eqnarray}

\item[Step 3]
Applying the Algorithm AB of Radu to $\{z, z_1, z_2, z_3, z_4\}$,
we find that $1$ is a $z$-module basis of $\langle GE^{\infty}(10)\rangle_{\mathbb{Q}}$. Thus
\begin{align}\label{GE10}
\langle GE^{\infty}(10)\rangle_{\mathbb{Q}}=\left\langle1\right\rangle_{\mathbb{Q}[z]}.
\end{align}

\item[Step 4]
We obtain that
\begin{align}\label{h-over}
  h=\frac{z_1^2z_3^3z_4^3}{z^6z_2^4}=\frac{\eta^{11}(\tau)
\eta_{5,1}^{12}(\tau)\eta^{15}(10\tau)}
  {\eta^{7}(2\tau)\eta^{19}(5\tau)\eta_{10,1}^{14}(\tau)},
\end{align}
for which $hF$ has a pole only at infinity.

\item[Step 5]
Applying Radu's Algorithm MW to $hF$, $z$ and $1$, we see that $hF\in \langle GE^\infty(10)\rangle_{\mathbb{Q}}$ and
\begin{align}\label{simple-over5n2-1}
hF=4z^3+4z^2-32z+32.
\end{align}
\end{description}

The relation \eqref{simple-over5n2-1} can be restated as the following theorem.
The implementations of the above steps will be described in the
subsequent sections.

\begin{thm}\label{thm-over-eq}
We have
\begin{align}\label{thm-over-eta-form}
y\sum_{n=0}^{\infty}\overline{p}(5n+2)q^n
=4z^3+4z^2-32z+32,
\end{align}
where
\begin{align*}
y &= \frac{(q;q)_{\infty}^{11}(q^{10};q^{10})_{\infty}^{16}
 (q,q^4;q^5)_{\infty}^{12}(q^5;q^{10})_{\infty}^{18}}
 {q^3(q^2;q^2)_{\infty}^{7}(q^5;q^5)_{\infty}^{19}
 (q,q^9;q^{10})_{\infty}^{14}
 (q^4,q^6;q^{10})_{\infty}^8},\\[6pt]
z&=
\frac{(q;q)_\infty(q^5;q^5)_\infty}
{q(q,q^4;q^5)_\infty^2(q^{10};q^{10})_\infty^2
(q,q^9;q^{10})_\infty}.
\end{align*}
\end{thm}

\section{Generators of $GE^\infty(N)$}\label{generator}

In this section, we show how to implement Step 1 as in the sketch of the previous
section,
that is, finding a finite set of generators of $GE^\infty(N)$.

In light of the symmetry
$$\eta_{\delta, g}(\tau) = \eta_{\delta, \delta-g}(\tau),$$
for any $\delta > 0$ and $\lfloor\delta/2\rfloor < g \leq \delta$,
we may rewrite the generalized eta-quotient $h(\tau)$ in $GE^\infty(N)$
in the following form
\begin{align}\label{gen-ele}
\prod\limits_{\delta|N\atop 0\leq g \leq \lfloor\delta/2\rfloor}\eta_{\delta, g}^{a_{\delta, g}}(\tau),
\end{align}
where
\begin{align}\label{adgg1}
 a_{\delta,g}\in
\begin{cases}
 \frac{1}{2}\mathbb{Z}, &\text{if}\;g=0\;\text{or}\;g=\frac{\delta}{2},\\[6pt]
\mathbb{Z}, &\text{otherwise}.
\end{cases}
\end{align}
Throughout this section, we assume that the generalized eta-quotients
are of the form \eqref{gen-ele}.

To find a set of generators of $GE^\infty(N)$, we first give a
 characterization of generalized eta-quotients $h(\tau)$ in $GE^\infty(N)$,
which involves the orders of $h(\tau)$ at all cusps of $\Gamma_1(N)$.
For any cusp $s$ of $\Gamma_1(N)$,
in order to apply a  formula of Robins \cite[Theorem~4]{Robins-1994} to
compute the order of $h(\tau)$ at   a cusp $s$,
we need to find a cusp of the form $\frac{\lambda}{\mu\varepsilon}$  that is equivalent to $s$,
where $\varepsilon|N$ and
\begin{align}\label{cusp-eq-cond}
  \gcd(\lambda, N)=\gcd(\lambda, \mu)=\gcd(\mu, N)=1.
\end{align}
The existence of such a cusp in the above form  is ensured
 by Corollary 4 of Cho, Koo and Park \cite{Cho-Koo-Park-2009}.

The following theorem gives a characterization of generalized eta-quotients
in $GE^\infty(N)$.

\begin{thm}\label{determine-GE}
Let
$$\mathcal{S}(N)=\{s_1, s_2,\ldots, s_\epsilon\}$$
be a complete set of inequivalent cusps of $\Gamma_1(N)$ and $s_\epsilon=\infty$.
Assume that for any $1\leq i\leq \epsilon$, $s_i$ is equivalent to $\frac{\lambda_i}{\mu_i\varepsilon_i}$,
where $\varepsilon_i|N$ and
\begin{align}\label{cusp-eq-cond}
  \gcd(\lambda_i, N)=\gcd(\lambda_i, \mu_i)=\gcd(\mu_i, N)=1.
\end{align}
Then a generalized eta-quotient $h(\tau)$ in the form of \eqref{gen-ele}
is in $GE^\infty(N)$ if and only if
the following conditions hold:
\begin{align}\label{expon-cond}
\begin{cases}
\sum\limits_{\delta|N}a_{\delta,0}=0,\\[6pt]
\frac{N}{2}\sum\limits_{\delta|N\atop 0\le g \le \lfloor\delta/2\rfloor}\frac{\gcd^2(\delta, \varepsilon_1)}{\delta \varepsilon_1}P_2\Big(\frac{\lambda_1 g}{\gcd(\delta, \varepsilon_1)}\Big)a_{\delta,g} \in\mathbb{N},\\
\qquad\vdots\\
\frac{N}{2}\sum\limits_{\delta|N\atop 0\le g \le \lfloor\delta/2\rfloor}\frac{\gcd^2(\delta, \varepsilon_{\epsilon-1})}{\delta \varepsilon_{\epsilon-1}}P_2\Big(\frac{\lambda_{\epsilon-1} g}{\gcd(\delta, \varepsilon_{\epsilon-1})}\Big)a_{\delta,g}\in\mathbb{N},\\[6pt]
\frac{N}{2}\sum\limits_{\delta|N\atop 0\le g \le \lfloor\delta/2\rfloor}\frac{\gcd^2(\delta, \varepsilon_\epsilon)}{\delta \varepsilon_{\epsilon}}P_2\Big(\frac{\lambda_\epsilon g}{\gcd(\delta, \varepsilon_\epsilon)}\Big)a_{\delta,g}\in\mathbb{Z}.
\end{cases}
\end{align}
\end{thm}

\begin{proof}
Assume that the generalized eta-quotient $h(\tau)$ as given by
\eqref{gen-ele} is in $GE^\infty(N)$.
By the transformation formula
of Schoeneberg \cite[p.~199~(30)]{Schoeneberg-1974} for $\eta_{g,h}^{(s)}(\tau)$,
we have
\begin{align}\label{weight_0_general}
\sum_{\delta|N}a_{\delta,0}=0,
\end{align}
and so the first condition in \eqref{expon-cond} is satisfied.
To show that the remaining conditions in \eqref{expon-cond} are satisfied,
we proceed to compute the order  of $h(\tau)$ at each cusp
 in $\mathcal{S}(N)$.
Since $h(\tau)\in GE^\infty(N)$,
 for all $1\leq i\leq \epsilon-1$,
\begin{align}\label{order-cond}
\mathrm{ord}_{s_i}(h(\tau))\in\mathbb{N}
\end{align}
and
\begin{align}\label{order-cond-infty}
  \mathrm{ord}_{s_\epsilon}(h(\tau))\in\mathbb{Z}.
\end{align}
For any $1\leq i\leq \epsilon$,
since $s_i$ is equivalent to $\frac{\lambda_i}{\mu_i\varepsilon_i}$,
we get
\begin{align*}
\mathrm{ord}_{s_i} (h(\tau)) = \mathrm{ord}_{\lambda_i/\mu_i\varepsilon_i} (h(\tau)).
\end{align*}
Using the formula of Robins \cite[Theorem~4]{Robins-1994} for the order of $h(\tau)$  at the cusp ${\lambda_i/\mu_i\varepsilon_i}$, namely,
\begin{align*}
\mathrm{ord}_{\lambda_i/\mu_i\varepsilon_i} (h(\tau))
=
\frac{N}{2}\sum\limits_{\delta|N\atop 0\le g \le \lfloor\delta/2\rfloor}\frac{\gcd^2(\delta, \varepsilon_i)}{\delta \varepsilon_i}P_2\Big(\frac{\lambda_i g}{\gcd(\delta, \varepsilon_i)}\Big)a_{\delta,g},
\end{align*}
we find that
\begin{align}\label{order_formula_eq}
\mathrm{ord}_{s_i} (h(\tau))
=
\frac{N}{2}\sum\limits_{\delta|N\atop 0\le g \le \lfloor\delta/2\rfloor}\frac{\gcd^2(\delta, \varepsilon_i)}{\delta \varepsilon_i}P_2\Big(\frac{\lambda_i g}{\gcd(\delta, \varepsilon_i)}\Big)a_{\delta,g}.
\end{align}
For $1\leq i\leq \epsilon-1$, combining \eqref{order-cond} and \eqref{order_formula_eq},
we obtain that
\begin{align}\label{order-cusp}
\frac{N}{2}\sum\limits_{\delta|N\atop 0\le g \le \lfloor\delta/2\rfloor}\frac{\gcd^2(\delta, \varepsilon_i)}{\delta \varepsilon_i}P_2\Big(\frac{\lambda_i g}{\gcd(\delta, \varepsilon_i)}\Big)a_{\delta,g}\in\mathbb{N}.
\end{align}
Setting $i=\epsilon$ in \eqref{order_formula_eq},
it follows from \eqref{order-cond-infty} that
\begin{align}\label{order-infty}
\frac{N}{2}\sum\limits_{\delta|N\atop 0\le g \le \lfloor\delta/2\rfloor}\frac{\gcd^2(\delta, \varepsilon_\epsilon)}{\delta \varepsilon_{\epsilon}}P_2\Big(\frac{\lambda_\epsilon g}{\gcd(\delta, \varepsilon_\epsilon)}\Big)a_{\delta,g}\in\mathbb{Z}.
\end{align}
Combining \eqref{weight_0_general},
\eqref{order-cusp} and \eqref{order-infty},
we are led to \eqref{expon-cond}.

Conversely,
assume that the conditions in \eqref{expon-cond} are satisfied.
From \eqref{expon-cond} and \eqref{order_formula_eq},
we see that
\begin{align}\label{Robins-con1}
\mathrm{ord}_0(h(\tau)) \in \mathbb{Z} \quad \text{and} \quad
\mathrm{ord}_\infty(h(\tau)) \in \mathbb{Z}.
\end{align}
The first condition of \eqref{expon-cond}
says that
\begin{align}\label{Robins-con2}
\sum_{\delta|N}a_{\delta,0}=0.
\end{align}
Robins \cite{Robins-1994} showed that if
a generalized eta-quotient $h(\tau)$ satisfies \eqref{Robins-con1} and
\eqref{Robins-con2}, 
then for any $\gamma\in\Gamma_1(N)$,
\begin{align}\label{h_trans}
h(\gamma\tau) = h(\tau).
\end{align}
By \eqref{order_formula_eq} and the conditions in \eqref{expon-cond},
we see that
for any $s\in\mathcal{S}(N)\setminus\{ \infty\}$,
\begin{align}\label{h_order}
\mathrm{ord}_s(h(\tau))\in\mathbb{N}.
\end{align}
Combining \eqref{h_trans} and \eqref{h_order}, we conclude that $h(\tau)\in GE^\infty(N)$.
\end{proof}

Based on the above theorem, the generalized eta-quotients in $GE^\infty(N)$
are determined by the solutions of \eqref{expon-cond}.
Next we show that  \eqref{expon-cond} can be solved by
transforming the conditions in \eqref{expon-cond}   to
a system of Diophantine inequalities, so that we can
obtain a finite set of generators of $GE^\infty(N)$.

Set
$$y_i = \frac{N}{2}\sum\limits_{\delta|N\atop 0\le g \le \lfloor\delta/2\rfloor}\frac{\gcd^2(\delta, \varepsilon_i)}{\delta \varepsilon_i}P_2\Big(\frac{\lambda_i g}{\gcd(\delta, \varepsilon_i)}\Big)a_{\delta,g}$$
for $1\leq i \leq \epsilon$.
It follows from \eqref{expon-cond} that $y_i\in\mathbb{N}$ for $1\leq i \leq \epsilon-1$
and $y_\epsilon\in\mathbb{Z}$.
Let
\begin{align*}
\chi_\delta(g)=
\begin{cases}
2, &\text{if}\;g=0\;\text{or}\;g=\frac{\delta}{2},\\[6pt]
1, &\text{otherwise},
\end{cases}
\end{align*}
and $a_{\delta,g}'=\chi_\delta(g) \, a_{\delta,g}$
for any $\delta|N$ and $0\leq g \leq \lfloor \delta/2\rfloor$.
By \eqref{adgg1}, it can be easily checked that each
$a_{\delta, g}'$ is an integer.
Then by Theorem \ref{determine-GE},
$h(\tau)\in GE^\infty(N)$ if and only if
$a_{\delta, g}'$ $(\delta|N, 0\le g\le\lfloor\delta/2\rfloor)$ and
$y_i$ $(1\le i \le \epsilon)$ is an integer solution of the following Diophantine inequalities:
\begin{align}\label{linear_inequalities}
\begin{cases}
\sum\limits_{\delta|N}a_{\delta, 0}'=0,\\[6pt]
\frac{N}{2}\sum\limits_{\delta|N\atop 0\le g \le \lfloor\delta/2\rfloor}\frac{\gcd^2(\delta, \varepsilon_1)}{\delta \varepsilon_1}P_2\Big(\frac{\lambda_1 g}{\gcd(\delta, \varepsilon_1)}\Big)\frac{a_{\delta,g}'}{\chi_\delta(g)}-y_1=0,\\[6pt]
\quad\vdots\\[6pt]
\frac{N}{2}\sum\limits_{\delta|N\atop 0\le g \le \lfloor\delta/2\rfloor}\frac{\gcd^2(\delta, \varepsilon_{\epsilon-1})}{\delta \varepsilon_{{\epsilon-1}}}P_2\Big(\frac{\lambda_{\epsilon-1} g}{\gcd(\delta, \varepsilon_{\epsilon-1})}\Big)\frac{a_{\delta,g}'}{\chi_\delta(g)}-y_{\epsilon-1}=0,\\[6pt]
\frac{N}{2}\sum\limits_{\delta|N\atop 0\le g \le \lfloor\delta/2\rfloor}\frac{\gcd^2(\delta, \varepsilon_\epsilon)}{\delta \varepsilon_\epsilon}P_2\Big(\frac{\lambda_\epsilon g}{\gcd(\delta, \varepsilon_\epsilon)}\Big)\frac{a_{\delta,g}'}{\chi_\delta(g)}-y_\epsilon=0,\\[6pt]
y_1\ge 0,\\[6pt]
\quad\vdots\\[6pt]
y_{\epsilon-1}\ge 0.
\end{cases}
\end{align}
Notice that different cusps may have the same order for $h(\tau)$,
there may exist redundant relations in above system of relations.
More precisely,
if for two cusps $s_i,s_j\in\mathcal{S}(N)\setminus\{\infty\}$,
$$\mathrm{ord}_{s_i}(h(\tau))=\mathrm{ord}_{s_j}(h(\tau)),$$
then we may ignore the relations contributed by $s_j$.
We now assume that
after the elimination of redundant relations,
the remaining relations
are still in the same form as in \eqref{linear_inequalities}.
It is known that there exist integral vectors
$\alpha_1, \ldots, \alpha_{k}$
such that the set of integer solutions of \eqref{linear_inequalities} is given by
\[\{u_1\alpha_1+\cdots+u_{k}\alpha_{k}\colon u_1, \ldots, u_{k}\in\mathbb{N}\},\]
see \cite[p.~234]{Schrijver-1986},
which implies that $GE^\infty(N)$ has a finite set of generators $z_1, \ldots, z_{k}$.
One can use the
package \texttt{4ti2} \cite{4ti2} in SAGE to find such a set of integral vectors $\alpha_1, \ldots, \alpha_{k}$.

Let us consider the case $N=10$ as an example.
Notice that for any generalized eta-quotient $h(\tau)$,
$$\mathrm{ord}_{1/4}(h(\tau))=\mathrm{ord}_{1/2}(h(\tau))$$
and
$$\mathrm{ord}_{0}(h(\tau))=\mathrm{ord}_{1/3}(h(\tau)).$$
By \eqref{linear_inequalities},
we obtain the following  Diophantine inequalities
after eliminating the relations contributed by the cusps $1/2$ and $1/3$:
\begin{align}\label{linear_inequalities-10}
\begin{cases}
a_{1,0}'+a_{2,0}'+a_{5,0}'+a_{10,0}'=0,\\[6pt]
\frac{5 \,a_{1,0}'}{12}
+\frac{5 \,a_{2,0}'}{24}
  +\frac{5\,a_{2,1}'}{24}
+\frac{a_{5,0}'}{12}+
  \frac{a_{5,1}}{6}+\frac{a_{5,2}}{6}\\[6pt]
  \quad+\frac{a_{10,0}'}{24}+\frac{a_{10,1}}{12}
  +\frac{a_{10,2}}{12}+\frac{a_{10,3}}{12}
  +\frac{a_{10,4}}{12}+\frac{a_{10,5}'}{24}-y_1=0,\\[6pt]
\quad\vdots\\[6pt]
\frac{5\,a_{1,0}'}{24}+\frac{5\,a_{2,0}'}{12}
-\frac{5 \,a_{2,1}'}{24}
+\frac{a_{5,0}'}{24}
+\frac{a_{5,1}'}{12}
+\frac{a_{5,2}'}{12}\\[6pt]
\quad+\frac{a_{10,0}'}{12}
-\frac{a_{10,1}'}{12}
+\frac{a_{10,2}'}{6}
-\frac{a_{10,3}'}{12}
+\frac{a_{10,4}'}{6}-\frac{a_{10,5}'}{24}-y_5 = 0,
\\[6pt]
\frac{a_{1,0}'}{24}+\frac{a_{2,0}'}{12}-\frac{a_{2,1}'}{24}
+\frac{5 \,a_{5,0}'}{24}+\frac{a_{5,1}'}{60}-\frac{11 \,a_{5,2}'}{60}\\[6pt]
\quad
+\frac{5 \,a_{10,0}'}{12}+\frac{23 \,a_{10,1}'}{60}
+\frac{a_{10,2}'}{30}-\frac{13 \,a_{10,3}'}{60}
-\frac{11 \,a_{10,4}'}{30}-\frac{5 \,a_{10,5}'}{24}-y_6=0\\[6pt]
y_1\ge 0,\\[6pt]
\quad\vdots\\[6pt]
y_{5}\ge 0.
\end{cases}
\end{align}
Each solution
$(a_{1,0}', \ldots, a_{10,5}', y_1,\ldots, y_6)$ of \eqref{linear_inequalities-10}
can be expressed as
\begin{align}\label{sol-10}
  \sum_{i=1}^5 c_i \alpha_i+\sum_{i=1}^6 d_i \beta_i,
\end{align}
where $c_1,\ldots,c_5$
are nonnegative integers,
$d_1,\ldots, d_{6}$ are integers and
\begin{align*}
&\alpha_1=(-1, 2, 0, 1, 2, 0, -2, -4, 0, 0, 0, 0, 0, 1, 0, 0, 0, -2),\\[6pt] &\alpha_2=(-1, -1, 0, 3, 4, 0, -1, -3, 0, 0, 0, 0, 0, 0, 1, 0, 0, -1),\\[6pt]
&\alpha_3=(1, -2, 0, -1, 2, 0, 2, -4, 0, 0, 0, 0, 0, 0, 0, 0, 1, -1),\\[6pt]
&\alpha_4=(1, 0, 0, 1, -2, 0, -2, -1, 0, 0, 0, 0, 0, 0, 0, 1, 0, -1),\\[6pt]
&\alpha_5 =(4, -3, 0, 0, 2, 0, -1, -4, 0, 0, 0, 0, 1, 0, 0, 0, 0, -2),\\[6pt]
&\beta_1=(0, 0, 0, -1, 0, 0, 1, 0, 0, 0, 0, 1, 0, 0, 0, 0, 0, 0),\\[6pt]
&\beta_2=(-1, 1, 1, 0, 0, 0, 0, 0, 0, 0, 0, 0, 0, 0, 0, 0, 0, 0),\\[6pt]
&\beta_3=(-1, 0, 0, 1, 1, 1, 0, 0, 0, 0, 0, 0, 0, 0, 0, 0, 0, 0),\\[6pt]
&\beta_4=(0, -1, 0, 0, 1, 0, 1, -1, 1, 0, 0, 0, 0, 0, 0, 0, 0, 0),\\[6pt]
&\beta_5=(-1, 1, 0, 1, 0, 0, -1, 1, 0, 1, 0, 0, 0, 0, 0, 0, 0, 0),\\[6pt]
&\beta_6=(0, 0, 0, 0, -1, 0, 0, 1, 0, 0, 1, 0, 0, 0, 0, 0, 0, 0).
\end{align*}
Since $a_{\delta,g}= a_{\delta,g}'/\chi_\delta(g)$,
we obtain eleven generalized eta-quotients.
It can be checked that the generalized eta-quotients corresponding to
$\beta_1, \ldots, \beta_6$ are equal to 1.
For example, the generalized eta-quotient corresponding to
$\beta_1$ is given by
\begin{align}
  h(\tau)=
  \frac{\eta_{10,0}^{\frac{1}{2}}(\tau)\eta_{10,5}^{\frac{1}{2}}(\tau)}{\eta_{5,0}^{\frac{1}{2}}(\tau)}.
\end{align}
Invoking \eqref{ge-e}, namely,
\begin{align*}
  \eta_{\delta, 0}(\tau)=\eta^2(\delta\tau)\quad\text{and}\quad\eta_{\delta, \frac{\delta}{2}}(\tau)=\frac{\eta^2(\frac{\delta}{2}\tau)}{\eta^2(\delta\tau)}.
\end{align*}
we obtain that $h(\tau)=1$.
The generalized eta-quotients corresponding to
$\alpha_1, \ldots, \alpha_5$ are the generators $z_1, z_2, z_3,z, z_4$
 as given in \eqref{generator_10}.

\section{Radu's Algorithm AB}\label{Radu-AB}

In the previous section, it was shown that
$GE^\infty(N)$ admits a finite set of generators $z_1, \ldots, z_{k}$.
Radu \cite{Radu-2009} developed the Algorithm AB to
produce a module basis of $\langle E^\infty(N)\rangle_{\mathbb{Q}}$,
based on a finite set of generators of $E^\infty(N)$.
In this section, we demonstrate how to apply
Radu's Algorithm AB to a finite set of
generators of $GE^\infty(N)$ to derive
a modular function $z$ and a module basis
$1, e_1, \ldots, e_w$ of the $\mathbb{Q}[z]$-module
$\langle GE^\infty(N)\rangle_{\mathbb{Q}}$.

We first give an overview of Radu's Algorithm AB.
Given  modular functions $z_1, \ldots, z_k$ for $\Gamma_0(N)$ with poles only at infinity,
Radu's Algorithm AB aims to
produce a modular function $z\in \mathbb{Q}[z_1, \ldots, z_k]$
and a $z$-reduced sequence
$e_1, \ldots, e_w\in \mathbb{Q}[z_1, \ldots, z_k]$
such that
\begin{align}\label{basis-ration}
\mathbb{Q}[z_1, \ldots, z_k] = \mathbb{Q}[z]+\mathbb{Q}[z] e_1+\cdots +\mathbb{Q}[z]e_w.
\end{align}
The condition on a  $z$-reduced sequence
ensures that $1, e_1, \ldots, e_w$ form a $\mathbb{Q}[z]$-module basis of
$\mathbb{Q}[z_1, \ldots, z_k]$.
The right hand side of \eqref{basis-ration} is denoted by
$\langle 1, e_1, \ldots, e_w\rangle_{\mathbb{Q}[z]}$.

Let $\left\langle E^\infty(N)\right\rangle_{\mathbb{Q}}$ denote the vector space over $\mathbb{Q}$ generated by $E^\infty(N)$.
As pointed out by Radu \cite{Radu-2015},
$\left\langle E^\infty(N)\right\rangle_{\mathbb{Q}}$ does not
have a finite basis as a vector space over $\mathbb{Q}$,
but it has a finite basis when considered as a $\mathbb{Q}[z]$-module for some $z$ in $\left\langle E^\infty(N)\right\rangle_{\mathbb{Q}}$.
To obtain such a modular function $z$ and a $\mathbb{Q}[z]$-module basis,
Radu applied the Algorithm AB to the generators $z_1, \ldots, z_k$ of $E^\infty(N)$,
then obtained a $z$-module basis $1, e_1, \ldots, e_w$ of the $\mathbb{Q}[z]$-module $\left\langle E^\infty(N)\right\rangle_{\mathbb{Q}}$ for some $z\in \left\langle E^\infty(N)\right\rangle_{\mathbb{Q}}$.

As will be seen, Radu's Algorithm AB can be adapted to $\Gamma_1(N)$.
The output of Algorithm AB consists of a modular function
$z\in \mathbb{Q}[z_1, \ldots, z_k]$
and a $z$-reduced sequence
$e_1,\ldots,e_w$.
The output of the Algorithm AB will be
carried over to the Algorithm MC and the Algorithm MW,
which require the input of a $z$-reduced sequence.
Thus, for the purpose of this paper,
we do not need to elaborate on the definition of a $z$-reduced sequence,
which can be found in \cite{Radu-2015}.

It is known that if $f$ is a modular function for $\Gamma_0(N)$
such that  ${\rm ord}_{a/c}(f)$ $\geq 0$
for every cusp $a/c$ of $\Gamma_0(N)$,
then $f$ is a constant,
see Newman \cite[Section, Proof of Lemma 3]{Newman-1957},
Knopp \cite[Chapter~2, Theorem~7]{Knopp-1993},
and Radu \cite[Lemma~5]{Radu-2015}.
Notice that this assertion also holds for $\Gamma_1(N)$.
Thus the Algorithm AB applies to  modular functions with poles only at infinity for $\Gamma_1(N)$.
It is worth mentioning that the Algorithm AB is based on the algorithms MC,
VB and MB,
which are also valid for modular functions
with poles only at infinity for $\Gamma_1(N)$.
Since the Algorithm MW of Radu is a refinement of the Algorithm MC,
it also works for $\Gamma_1(N)$.

We proceed to find a modular function $z$
and a module basis of $\mathbb{Q}[z]$-module
$\langle GE^\infty(N)\rangle_\mathbb{Q}$.
Let $\{z_1, \ldots, z_{k}\}$ be a finite set of generators of $GE^\infty(N)$.
Note that
\begin{align}\label{generator-GE}
\langle GE^\infty(N)\rangle_\mathbb{Q} = \mathbb{Q}[z_1, \ldots, z_k].
\end{align}
Applying the Algorithm AB to $z_1, z_2, \ldots, z_k$,
we obtain a modular function
$z\in \left\langle GE^\infty(N)\right\rangle_{\mathbb{Q}}$ and a $z$-reduced
sequence $e_1, \ldots, e_w\in \left\langle GE^\infty(N)\right\rangle_{\mathbb{Q}}$
such that
\begin{align}\label{AB-GE}
\mathbb{Q}[z_1, \ldots, z_k]=
\langle 1, e_1, \ldots, e_w\rangle_{\mathbb{Q}[z]}.
\end{align}
Combining \eqref{generator-GE} and \eqref{AB-GE}, we find that
\[\left\langle GE^\infty(N)\right\rangle_{\mathbb{Q}}=\langle 1, e_1, \ldots, e_w\rangle_{\mathbb{Q}[z]}.\]
Using the property that $e_1, e_2, \ldots, e_w$ form
a $z$-reduced sequence,
we deduce that $1, e_1, \ldots, e_w$ constitute a $\mathbb{Q}[z]$-module basis of $\langle GE^\infty(N)\rangle_\mathbb{Q}$.

For example,
applying the Algorithm AB for $\Gamma_1(N)$ to the generators
$ z, z_1, z_2, z_3, z_4 $ of $GE^\infty(10)$
given by \eqref{generator_10},
we obtain that
\begin{align}\label{GE10}
\langle GE^{\infty}(10)\rangle_{\mathbb{Q}}={\mathbb{Q}[z]}.
\end{align}

\section{Finding a Generalized Eta-Quotient}\label{computing_F}

In this section,
we present an implementation of Step 4 in the algorithm outlined in
Sect. \ref{Sketch}.
Assume that  $\{z_1, z_2, \ldots, z_k\}$  is a set of generators of $GE^\infty(N)$
and $F(\tau)$ is a modular function for $\Gamma_1(N)$ as given in \eqref{F_expression}.
Our objective is to
find a generalized eta-quotient $h(\tau)$ of the form
\begin{align}\label{h-form}
  h(\tau)=\prod_{j=1}^{k} z_j^{t_j},
\end{align}
such that the modular function $hF$ has a pole only at infinity,
that is, for any cusp $s\neq \infty$,
\begin{align}\label{ord-eq}
\mathrm{ord}_s(hF)\geq 0,
\end{align}
where $t_j$ are integers.
To find the integers $t_j$ for which the relation \eqref{ord-eq} holds,
we shall establish a system of linear inequalities any solution of which leads to a desired generalized eta-quotient $h$. The linear
inequalities are derived by the lower bounds of $\mathrm{ord}_s(hF)$
for all cusps $s\neq \infty$.

Now we utilize Theorem \ref{order_F-1}
to obtain the lower bound of $\mathrm{ord}_s(hF)$.
Let
$$\mathcal{S}(N)=\{s_1, s_2,\ldots, s_\epsilon\}$$
be a complete set of inequivalent cusps
of $\Gamma_1(N)$ and $s_\epsilon = \infty$.
For any $1\leq i \leq \epsilon$ and $1\leq j\leq k$, denote
$\mathrm{ord}_{s_i} z_j$ by $ b_{ij}.$
By the definition \eqref{h-form},
we have for each cusp $s_i$,
\begin{align}\label{order-hF}
  {\rm ord}_{s_i}(hF)=\sum_{j=1}^k t_j b_{ij}+{\rm ord}_{s_i}(F).
\end{align}
By Theorem \ref{order_F-1},
we see that for any $1\leq i\leq \epsilon$,
\begin{align}\label{low_ord_F}
    \mathrm{ord}_{s_i}(F(\tau))\ge d_i,
\end{align}
where
$$d_i = w_{\alpha_i}\,(p(\alpha_i)+p^*(\alpha_i)),$$
and $\alpha_i$ is defined in Theorem \ref{order_F-1}.
Combining \eqref{order-hF} and \eqref{low_ord_F}, we get
\begin{align}\label{lower-hF}
  {\rm ord}_{s_i}(hF)\geq \sum_{j=1}^k t_j b_{ij}+d_i.
\end{align}
Consider the Diophantine inequalities
\begin{align}\label{h-exist-1}
\begin{cases}
\sum\limits_{j=1}^k t_j b_{1j}+d_1 >-1,\\[3pt]
\quad\quad\vdots\\[3pt]
\sum\limits_{j=1}^k t_j b_{(\epsilon-1)j}+d_{\epsilon-1} >-1.
\end{cases}
\end{align}
Now,  if we can find integers $t_1, \ldots, t_k$ such that \eqref{h-exist-1} holds, then  \eqref{lower-hF} implies that
the generalized eta-quotient $h(\tau)$ determined by $z_1, z_2, \ldots, z_k$ and $t_1, t_2, \ldots, t_k$
satisfies \eqref{ord-eq}.
Hence we deduce that any integer solution of \eqref{h-exist-1}
leads to a generalized eta-quotient $h(\tau)$
such that $hF$ has a pole only at infinity.

We note that different generalized eta-quotients $h$
may lead to different expressions for $F$.
In order to get a relatively simple
expression for $F$,
we impose a further condition that the order of $hF$ at infinity
is as large as possible.
While we cannot rigorously describe what a simple expression means,
intuitively speaking,
the above
condition appears to play a role in getting a relatively simple
expression for $F$.
Next we state how to find such a generalized eta-quotient $h(\tau)$.

It is known that there exist integral vectors
$\alpha_1, \ldots, \alpha_{w}, \beta_1, \ldots, \beta_l$
such that the set of integer solutions of \eqref{h-exist-1} is given by
\begin{align}\label{sol-h}
\{\alpha_i +v_1\beta_1+\cdots + v_l\beta_l\colon 1\leq i\leq w \;\text{and} \;
v_1, \ldots, v_l\in\mathbb{N}\},
\end{align}
see \cite[p.~234]{Schrijver-1986}.

The following theorem shows how to find a generalized eta-quotient
$h$ such that $\mathrm{ord}_\infty(hF)$ attains the maximum value
among all the $h$ satisfying \eqref{h-exist-1}.

\begin{thm}\label{sol-best-h}
For $1\leq i \leq w$,
let
\begin{align*}
    \alpha_i=(\alpha_{i1},\alpha_{i2},\ldots,\alpha_{ik}),
\end{align*}
as given in \eqref{sol-h}.
Let $h_i$ be the generalized eta-quotient determined
by $z_1, z_2, \ldots, z_k$ and $\alpha_i$,
that is
\begin{align}\label{h_i-def}
  h_i(\tau) =\prod_{j=1}^{k} z_j^{\alpha_{ij}}.
\end{align}
Assume that
\begin{align}\label{ord-ineq-2}
  \mathrm{ord}_\infty(h_1F)\geq \mathrm{ord}_\infty(h_iF)
\end{align}
for $2\leq i\leq w$.
For any integer solution $\mu=(\mu_1,\mu_2,\ldots,\mu_k)$
of \eqref{h-exist-1},
let $g$ be the generalized eta-quotient
\begin{align}
    g(\tau)=\prod_{j=1}^{k} z_j^{\mu_{j}}.
\end{align}
Then we have
\begin{align*}
\mathrm{ord}_\infty(h_1F) \geq \mathrm{ord}_\infty(gF).
\end{align*}
\end{thm}

\begin{proof}
By \eqref{sol-h},
there exist an integer $1\leq i \leq w$,
and nonnegative integers $v_1, \ldots, v_l$ such that
\begin{align}\label{uni-solutrion}
  \mu = \alpha_i +v_1\beta_1+\cdots + v_l\beta_l.
\end{align}
For $1\leq j \leq l$,
let
\begin{align*}
    \beta_j=(\beta_{j1},\beta_{j2},\ldots,\beta_{jk}).
\end{align*}
and let $f_j$ be the generalized eta-quotient defined by
\begin{align}\label{f_i-def}
  f_j(\tau) =\prod_{i=1}^{k} z_i^{\beta_{ji}}.
\end{align}
Combining \eqref{h_i-def}, \eqref{uni-solutrion} and \eqref{f_i-def},
we obtain that
\begin{align*}
g(\tau) = h_i\prod_{j=1}^l f_j^{v_j}.
\end{align*}
Thus,
\begin{align}\label{ord-ineq-1}
\mathrm{ord}_\infty(gF) = \mathrm{ord}_\infty(h_iF)
+
\sum_{j=1}^l v_j \mathrm{ord}_\infty(f_j).
\end{align}
Under the condition \eqref{ord-ineq-2},
it follows from \eqref{ord-ineq-1} that
\begin{align}\label{ord-ineq}
\mathrm{ord}_\infty(gF) \leq \mathrm{ord}_\infty(h_1F)
+
\sum_{j=1}^l v_j \mathrm{ord}_\infty(f_j).
\end{align}
We claim that for each $1 \leq j \leq l$,
\begin{align}\label{o-lem}
  {\rm ord}_{\infty}(f_j)\leq 0.
\end{align}
There are two cases.

Case 1. If $f_j(\tau)$ is a constant, then
$\mathrm{ord}_\infty(f_j)=0.$

Case 2.
If $f_j(\tau)$ is not a constant,
we shall show that
$\mathrm{ord}_\infty (f_j) < 0$.
Assume to the contrary that
$\mathrm{ord}_\infty (f_j)\geq 0$.
Since $f_j(\tau)$ is not a constant,
there exists a cusp $s\neq \infty$ such that
$\mathrm{ord}_s (f_j)<0$.
By the assumption \eqref{ord-eq},
we have
$\mathrm{ord}_s(h_1F) \geq 0$.
Let $d = \mathrm{ord}_s(h_1F)$.
By \eqref{sol-h},
we see that $\alpha_1+(d+1)\beta_j$ is a solution of \eqref{h-exist-1}.
It follows that the generalized eta-quotient
$f_j^{ d+1 } h_1$
satisfies \eqref{ord-eq},
and so
\begin{align}\label{new-h-l}
\mathrm{ord}_s(f_j^{d+1}h_1F)\geq 0.
\end{align}
However, since $\mathrm{ord}_s(f_j)<0$,
we have
\begin{align*}
\mathrm{ord}_s(f_j^{d+1}h_1F) = (d+1)\mathrm{ord}_s(f_j)+d <0,
\end{align*}
which contradicts \eqref{new-h-l}.
Thus we deduce that $\mathrm{ord}_\infty (f_j) < 0$,
as claimed.
Combining the above two cases,
we find that \eqref{o-lem} holds for each $1\leq j\leq l$.
In view of \eqref{ord-ineq},
we conclude that
\begin{align}
  \mathrm{ord}_\infty(gF)\leq \mathrm{ord}_\infty(h_1F),
\end{align}
and this completes the proof.
\end{proof}

For the overpartition function $\overline{p}(5n+2)$, we have found a modular function $F(\tau)$ for $\Gamma_1(10)$ as given in \eqref{Ftau-Q31}.
For the  generators
$z, z_1, z_2, z_3, z_4$ of
$GE^{\infty}(10)$
as given in \eqref{generator_10}, we obtain the following
system of linear inequalities \eqref{h-exist-1}:
\begin{align}\label{h-5n2}
\begin{cases}
t_5-3>-1,\\[6pt]
t_3+\frac{19}{5}>-1,\\[6pt]
t_2-2>-1,\\[6pt]
t_4-\frac{18}{5}>-1,\\[6pt]
t_1+\frac{27}{5}>-1.
\end{cases}
\end{align}
Each integer solution $(t_1, t_2, t_3, t_4, t_5)$ of \eqref{h-5n2}
can be expressed as
\begin{align}\label{sol-10}
\alpha_1+\sum_{i=1}^5 v_i \beta_i,
\end{align}
where $v_1, \ldots, v_5$
are nonnegative integers, and
\begin{align*}
\alpha_1 &= (-6, 2, -4, 3, 3),\\[3pt]
\beta_1 &= (1, 0, 0, 0, 0),\\[3pt]
\beta_2 &= (0, 1, 0, 0, 0),\\[3pt]
\beta_3 &= (0, 0, 1, 0, 0),\\[3pt]
\beta_4 &= (0, 0, 0, 1, 0),\\[3pt]
\beta_5 &= (0, 0, 0, 0, 1).
\end{align*}
The generalized eta-quotient corresponding to $\alpha_1$ is
\begin{align}\label{h-over}
h=\frac{z_1^2z_3^3z_4^3}{z^6z_2^4}=\frac{\eta^{11}(\tau)
\eta_{5,1}^{12}(\tau)\eta^{15}(10\tau)}
  {\eta^{7}(2\tau)\eta^{19}(5\tau)\eta_{10,1}^{14}(\tau)}
\end{align}
and $hF$ has a pole only at infinity.
Consider a different solution $\mu=\alpha_1+2\beta_2=(6,4,-4,3,3)$ of \eqref{h-5n2},
we get a generalized eta-quotient
\begin{align}\label{ex-another-choice}
  h'=\frac{z_1^4z_3^3z_4^3}{z^6z_2^{4}}=
  \frac{\eta^{9}(\tau)\eta_{5,1}^{16}(\tau)\eta^{11}(10\tau)}
  {\eta^{3}(2\tau)\eta^{17}(5\tau)\eta_{10,1}^{22}(\tau)}
\end{align}
and $h'F$ has a pole only at infinity.
The orders of $hF$ and $h'F$ at infinity are $-3$ and $-7$, respectively.
As will be seen in the next section,
the Ramanujan-type identity derived from $hF$ takes a simpler
form than that derived from $h'F$.

\section{Ramanujan-Type Identities}\label{R-T-identities}

Given a partition function $a(n)$ as defined by \eqref{radu-gf},
and integers $m$ and $t$,
let
\begin{align}\label{F_expression-1}
F(\tau) = \phi(\tau)\,g_{m,t}(\tau)
\end{align}
be a modular function as given in \eqref{F_expression},
where $\phi(\tau)$ is a generalized eta-quotient of the form \eqref{phi-form},
and
\begin{align*}
g_{m,t}(\tau)  = q^{\frac{t-\ell}{m}}\sum\limits_{n=0}^\infty a(mn+t)q^n,
\end{align*}
as given in \eqref{gmt}.

Assume that we have found  a generalized eta-quotient $h(\tau)$ such that $hF$ has a pole only at infinity.
In Sect. \ref{Radu-AB}, we derived  a modular function
$z\in \left\langle GE^\infty(N)\right\rangle_{\mathbb{Q}}$
and a $z$-reduced sequence $e_1,\ldots,e_w$ such that
\begin{align*}
\left\langle GE^\infty(N)\right\rangle_{\mathbb{Q}} = \mathbb{Q}[z]+\mathbb{Q}[z] e_1+\cdots +\mathbb{Q}[z]e_w.
\end{align*}
In this section,
we aim to derive an expression for $hF$ in terms of $z$ and the module basis
$1, e_1,\ldots, e_w$.
This leads to a Ramanujan-type identity for $a(mn+t)$.

We first adapt Radu's Algorithm MC,
original designed for $\Gamma_0(N)$,
to $\Gamma_1(N)$,
and apply it to $hF$, $z$ and $e_1,\ldots,e_w$ to determine
whether $hF$ belongs to $\left\langle GE^\infty(N)\right\rangle_{\mathbb{Q}}$.
By Radu \cite[Lemma~5]{Radu-2015},
the Algorithm MC   requires the non-positive parts of the $q$-expansion of
$hF$,
and finite parts of the $q$-expansions of $z$, and $e_1,\ldots,e_w$.
More precisely,
by \eqref{F_expression-1},
the non-positive parts of the $q$-expansion of $hF$ can be computed
via the generating function \eqref{radu-gf} of $a(n)$
and the $q$-expansions of $h(\tau)$ and $\phi(\tau)$.
If the algorithm confirms that
$hF\in \left\langle GE^\infty(N)\right\rangle_{\mathbb{Q}}$,
then we may utilize the $\Gamma_1(N)$ version of Algorithm MW
to express $hF$ as
\begin{align}\label{hf-expression}
  hF = p_0(z) +p_{1}(z) e_{1} + \cdots +p_w(z) e_w,
\end{align}
where $p_i(z)\in \mathbb{Q}[z]$ for $0\leq i\leq w$.

To this end,
we first utilize the Radu's Algorithm MC for $\Gamma_1(N)$
to determine whether $hF$ belongs to
$\left\langle GE^\infty(N)\right\rangle_{\mathbb{Q}}$.
Once we have confirmed that
$hF\in \left\langle GE^\infty(N)\right\rangle_{\mathbb{Q}}$,
we may utilize the Algorithm MW of Radu for $\Gamma_1(N)$
to derive a Ramanujan-type identity for $a(mn+t)$.

We now give an algorithmic derivation of the Ramanujan-type identity for
$\overline{p}(5n+2)$,  as stated in Theorem \ref{thm-over-eq}.

For $F$, $z$ and $h$ given in \eqref{Ftau-Q31}, \eqref{generator_10} and \eqref{h-over},
we have
\begin{align*}
  hF & = \frac{4}{q^3}+\frac{28}{q^2}+\frac{56}{q}+140+O(q),\\[6pt]
  z  & =\frac{1}{q}+2+2 q+q^2+O(q^3).
\end{align*}
Applying Radu's Algorithm MC to $hF$ and $z$,
we deduce that
$$hF\in \left\langle GE^\infty(10)\right\rangle_{\mathbb{Q}}.$$
With the input $hF$ and $z$, the Algorithm MW
yields
\begin{align}\label{simple-over5n2}
hF=4z^3+4z^2-32z+32.
\end{align}
Substituting $F$, $z$ and $h$ into \eqref{simple-over5n2},
we obtain the Ramanujan-type identity in Theorem \ref{thm-over-eq}.
However,
if we take $h'$ as given in \eqref{ex-another-choice},
then we get
\begin{align}\label{ex-complex}
  h'F
  =4 z^{7}-4 z^{6}-44 z^{5}+100z^{4}-20z^{3}-92z^{2}+32z+32.
\end{align}

In the same vain,
we obtain a Ramanujan-type identity for $\overline{p}(5n+3)$.

\begin{thm}\label{thm-over-eq2}
We have
\begin{align*}
  y\sum_{n=0}^{\infty}\overline{p}(5n+3)q^n
  =8 z^3-12 z^2+16 z-16,
\end{align*}
where $z$ is given in Theorem \ref{thm-over-eq} and
\begin{align*}
 y &=\frac{
  (q;q)_{\infty}^{12}(q^5;q^5)_{\infty}^{12}(q,q^9;q^{10})_{\infty}^{2}
(q^4,q^6;q^{10})_{\infty}^8}
{q^3(q^2;q^2)_{\infty}^{7}(q,q^4;q^5)_{\infty}^{6}(q^{10};q^{10})_{\infty}^{16}
(q^5;q^{10})_{\infty}^{14}}.
\end{align*}
\end{thm}

Notice that Theorem \ref{thm-over-eq} and Theorem \ref{thm-over-eq2} can be considered as witness identities for the following congruences of Hirschhorn and Sellers \cite{Hirschhorn-Sellers}:
\begin{align*}
  \overline{p}(5n+2)&\equiv 0 \pmod{4},\\[6pt]
  \overline{p}(5n+3)&\equiv 0 \pmod{4}.
\end{align*}

\section{A Witness Identity for $p(11n+6)$}\label{app_partition}

In this section,
we demonstrate how our algorithm gives rise to a witness identity for $p(11n+6)$.
We begin with an overview of the witness identities due to
Bilgici and Ekin \cite{Bilgici-Ekin-2014},
Radu \cite{Radu-2015} and Hemmecke \cite{Hemmecke-2018}.
Bilgici and Ekin \cite{Bilgici-Ekin-2014} used the method
of Kolberg to deduce the generating functions of $p(11n+t)$
for all $0\leq t \leq 10$. In particular, they obtained the following witness
identity:
\begin{align}\label{B-E-11n+6}
\sum_{n=0}^\infty p(11n+6)q^{n} &= 11x
(-x_{1}^{3} x_{4}-x_{2}^{3} x_{5}-x_{4}^{3} x_{2}-x_{3}^{3} x_{1}-x_{5}^{3} x_{3}-14 x_{1}^{2}x_{4} \notag\\
&\qquad-14 x_{2}^{2} x_{5}-14 x_{4}^{2} x_{2}-14 x_{3}^{2} x_{1}-14 x_{5}^{2} x_{3}-29 x_{1} x_{4}\notag\\[6pt]
&\qquad \left.-29 x_{2}x_{5}-29 x_{2} x_{4}-29 x_{1} x_{3}-29 x_{3} x_{5}+106\right),
\end{align}
where
\begin{align*}
x&= \frac{q^{4}(q^{11}; q^{11})_\infty^{11}}
{(q; q)_\infty^{12}},\\[6pt]
x_1 &= -\frac{(q^4,q^7; q^{11})_\infty^2 (q,q^{10}; q^{11})_\infty}{(q^2,q^9; q^{11})_\infty^2 (q^5,q^{6}; q^{11})_\infty}\\[6pt]
x_2 &= -\frac{(q^2,q^9; q^{11})_\infty^2 (q^5,q^{6}; q^{11})_\infty}{q(q,q^{10}; q^{11})_\infty^2 (q^3,q^{8}; q^{11})_\infty},\\[6pt]
x_3 &= \frac{q^2(q,q^{10}; q^{11})_\infty^2 (q^3,q^{8}; q^{11})_\infty}{(q^4,q^7; q^{11})_\infty (q^5,q^{6}; q^{11})_\infty^2},\\[6pt]
x_4 &= \frac{(q^4,q^7; q^{11})_\infty (q^5,q^{6}; q^{11})_\infty^2}{q(q^2,q^9; q^{11})_\infty (q^3,q^{8}; q^{11})_\infty^2},\\[6pt]
x_5 &= -\frac{(q^2,q^9; q^{11})_\infty (q^3,q^{8}; q^{11})_\infty^2}{(q^4,q^7; q^{11})_\infty^2 (q,q^{10}; q^{11})_\infty}.
\end{align*}
Using the Ramanujan--Kolberg algorithm,
Radu \cite{Radu-2015} derived a witness identity for $p(11n+6)$.
A set  $\{M_1, M_2, \ldots, M_7\}$ of generators of
$E^\infty(22)$ can be found in \cite{Radu-2015}.
For example,
\begin{align*}
  M_1 = \frac{\eta^7(\tau)\eta^3(11\tau)}{\eta^3(2\tau)\eta^7(22\tau)}.
\end{align*}
Let
\begin{align*}
  F =\frac{(q;q)^{10}_\infty(q^2;q^2)^{2}_\infty(q^{11};q^{11})^{11}_\infty}
  {q^{14}(q^{22};q^{22})^{22}_\infty}\sum_{n=0}^\infty p(11n+6)q^n.
\end{align*}
Radu showed that
\begin{align}\label{radu}
F &=11(98t^4+1263t^3+2877t^2+1019t-1997)\nonumber\\[6pt]
  &\qquad +11z_1(17t^2+490t^2+54t-871)\nonumber\\[6pt]
  &\qquad
  +11z_2(t^3+251t^2+488t-614),
\end{align}
where
\begin{align*}
  t&=\frac{3}{88}M_1+\frac{1}{11}M_2-\frac{1}{8}M_4,\\[6pt]
  z_1&=-\frac{5}{88}M_1+\frac{2}{11}M_2-\frac{1}{8}M_4-3,\\[6pt]
  z_2&=\frac{1}{44}M_1-\frac{3}{11}M_2+\frac{5}{4}M_4.
\end{align*}
Noting that $(1-q^n)^{11}\equiv 1-q^{11n} \pmod {11}$
and $(1-q^n)^{8}\equiv (1-q^{2n})^4 \pmod {8}$, we see that
\eqref{radu} implies the Ramanujan congruence for $p(11n+6)$.
Hemmecke \cite{Hemmecke-2018} generalized Radu's algorithm and derived the following witness identity:
\begin{align}\label{Hem}
  F & = 11^2\cdot 3068 M_7 + 11^2\cdot (3M_1+ 4236)M_6 \nonumber\\[6pt]
  &\qquad+ 11\cdot(285M_1+11\cdot 5972)M_5 + 11(1867M_1+11\cdot 2476)M_2\nonumber\\[6pt]
  &\qquad - \frac{11}{8}(M_1^3 + 1011\, M_1^2+11\cdot 6588M_1 + 11^2\cdot 10880)\nonumber\\[6pt]
  &\qquad+ \frac{11}{8}(M_1^2+11\cdot 4497\,M_1 + 11^2 \cdot3156)M_4.
\end{align}

We are now ready to give an algorithmic derivation of the identity for $p(11n+6)$ as stated in Theorem \ref{thm-p(11n+6)}.

\begin{proof}[Proof of Theorem \ref{thm-p(11n+6)}]
Notice that $N=11$ satisfies all the conditions \nameref{con_N_1}--\nameref{con_N_10}.
We proceed with the following steps.
\begin{description}
 \setlength{\parskip}{2ex}
\item[Step 1] By Theorem \ref{con_F_modular_function},
we find that
\[F(\tau) = q (q^{11}; q^{11})_\infty\,
\sum\limits_{n=0}^\infty p(11n+6)q^n\]
is a modular function for $\Gamma_1(11)$.

\item[Step 2]
Solving the system of
Diophantine inequalities \eqref{linear_inequalities} for $N = 11$,
we obtain a set of 27 generators of $GE^{\infty}(11)$ including
$z$ and $e$ as given in \eqref{z} and \eqref{e}.

\item[Step 3]
Applying Radu's Algorithm AB, we deduce that
$$\langle GE^\infty(11)\rangle_\mathbb{Q}=\langle 1,e\rangle_{\mathbb{Q}[z]}.$$

\item[Step 4]
By virtue of Theorem \ref{order_F-1} and Theorem \ref{sol-best-h},
we get
$$h=\frac{\eta^{24}(\tau)}{\eta^{24}(11\tau)
\eta_{11,1}^{28}(\tau)\eta_{11,2}^{16}(\tau)\eta_{11,3}^{12}(\tau)
\eta_{11,4}^{4}(\tau)}$$
for which  $hF$ has a pole only at infinity.

\item[Step 5]
Employing Radu's Algorithm MC and Algorithm MW,
we deduce that $hF\in \langle GE^\infty(11)\rangle_{\mathbb{Q}}$ and
\begin{eqnarray*}
\begin{split}
hF
&=11  z^{10}+121   z^{8}e+330   z^{9}-484   z^7e-990   z^{8}+484   z^6 e+792   z^7\\[6pt]
&\qquad-484   z^5e+44   z^6+1089   z^4 e-132   z^5-1452   z^3e-451   z^4\\[6pt]
&\qquad+968   z^2e+748   z^3-242   ze-429   z^2+77   z+11.
\end{split}
\end{eqnarray*}
\end{description}
This completes the proof.
\end{proof}

\section{Further Examples}\label{app_further}

In this section,
we derive Ramanujan-type identities on
the broken $2$-diamond partition function.
The notion of the broken $k$-diamond partitions was introduced by Andrews and Paule \cite{Andrews-Paule-2007} in their study of MacMahon's partition analysis.
The number of broken $k$-diamond partitions of $n$ is
denoted by $\Delta_k(n)$. They showed that the generating function of $\Delta_k(n)$ is given by
\begin{align*}
\sum_{n=0}^\infty\Delta_k(n)q^n
=\frac{(q^2;q^2)_\infty(q^{2k+1};q^{2k+1})_\infty}
{(q;q)^3_\infty(q^{4k+2};q^{4k+2})_\infty}.
\end{align*}
Andrews and Paule conjectured that
\begin{align}\label{Dia-con-1}
\Delta_2(25n+14)\equiv 0\pmod5.
\end{align}
Chan \cite{Chan-2008} proved this conjecture and also showed that
\begin{align}\label{Dia-con-2}
\Delta_2(25n+24)\equiv 0\pmod5.
\end{align}
Define $a(n)$ by
\[\sum_{n=0}^\infty a(n)q^n= \frac{(q;q)^2_\infty(q^2;q^2)_\infty}{(q^{10};q^{10})_\infty}.\]
Since $(1-q^n)^5\equiv 1-q^{5n} \pmod{5}$, we see that
$\Delta_2(n)\equiv a(n) \pmod 5$.
By the Ramanujan--Kolberg algorithm,
Radu \cite{Radu-2015} obtained the following identity:
\begin{align}\label{a-RK}
  \frac{(q^2;q^2)^{12}_\infty(q^5;q^5)^{10}_\infty}{q^4(q;q)^{6}_\infty(q^{10};q^{10})^{20}_\infty}
  &\left(\sum_{n=0}^{\infty} a(25 n+14) q^{n}\right)\left(\sum_{n=0}^{\infty} a(25 n+24) q^{n}\right)\notag\\[6pt]
  &\qquad=
  25\left(2 t^{4}+28 t^{3}+155 t^{2}+400 t+400\right),
\end{align}
where
\begin{align*}
t=\frac{(q;q)^3_{\infty } (q^5;q^5)_{\infty }}{q (q^2;q^2)_{\infty } (q^{10};q^{10})_{\infty }^3}.
\end{align*}
The congruences \eqref{Dia-con-1} and \eqref{Dia-con-2} are easy consequences of \eqref{a-RK}.
Let
\begin{align*}
  z=\frac{(q^2;q^2)_{\infty } (q^5;q^5)_{\infty }^5}{q (q;q)_{\infty } (q^{10};q^{10})_{\infty }^5}.
\end{align*}
Using the package \texttt{RaduRK}, Smoot \cite{Smoot-2019} deduced
that
\begin{align*}
  \frac{(q;q)^{126}_{\infty } (q^5;q^5)^{70}_{\infty }}{q^{58} (q^2;q^2)_{\infty }^2 (q^{10};q^{10})_{\infty }^{190}}
  \left(\sum_{n=0}^\infty\Delta_2(25n+14)q^n\right)
  \left(\sum_{n=0}^\infty\Delta_2(25n+24)q^n\right)
\end{align*}
is a polynomial in $z$ of degree $58$ with integer coefficients  divisible by 25.
It is not hard to see that the above relation implies
 the congruences \eqref{Dia-con-1} and \eqref{Dia-con-2}.

Our algorithm provides the following witness identities  for $\Delta_2(25n+14)$
and $\Delta_2(25n+24)$.

\begin{thm}\label{diamond-thm}
Let
\begin{align*}
z=\frac{(q;q)_\infty(q^5;q^5)_\infty}
{q(q,q^4;q^5)_\infty^2(q^{10};q^{10})_\infty^2
(q,q^9;q^{10})_\infty}.
\end{align*}
Then
\begin{align}\label{diamond-eq1}
\frac
{(q;q)_\infty^{92}(q^5;q^5)_\infty^{14}(q,q^4;q^5)_\infty^{52}
(q^4,q^6;q^{10})_\infty^{4}}
{q^{57}(q^2;q^2)_\infty^{58}(q^{10};q^{10})_\infty^{46}(q,q^9;q^{10})_\infty^{109}
(q^5;q^{10})_\infty^{10}}\,
\sum_{n=0}^\infty \Delta_{2}(25n+14)q^n
\end{align}
and
\begin{align}\label{diamond-eq2}
  \frac{(q;q)_\infty^{92}(q,q^4;q^5)_\infty^{62}(q^5;q^{10})_\infty^{6}}
  {q^{57}(q^2;q^2)_\infty^{59}(q^5;q^5)_\infty^{2}(q^{10};q^{10})_\infty^{29}
(q,q^9;q^{10})_\infty^{119}(q^4,q^6;q^{10})_\infty^{4}}\sum_{n=0}^\infty \Delta_{2}(25n+24)q^n
\end{align}
are both polynomials in $z$ of degree 57 with integer coefficients divisible by $5$.
\end{thm}
More precisely,  \eqref{diamond-eq1} equals
\begin{align*}
&\,\; 10445 z^{57}+65072505 z^{56}+29885191700 z^{55}+2909565072375 z^{54}\\[6pt]
&\quad+58232762317950 z^{53}-771909964270635 z^{52}-8976196273201590 z^{51}\\[6pt]
&\quad+168096305999838525 z^{50}-552704071429548750 z^{49}\\[6pt]
&\quad-6285133254753356625 z^{48}+76077164750182724400 z^{47}\\[6pt]
&\quad-350853605818104040400 z^{46}+430844106211910184000 z^{45}\\[6pt]
&\quad+4332665789140456020000 z^{44}-31965516977695010144000 z^{43}\\[6pt]
&\quad+116598487085627561478400 z^{42}-254498980254624708134400 z^{41}\\[6pt]
&\quad+226239786150985106784000 z^{40}+630144010340120712320000 z^{39}\\[6pt]
&\quad-3270835930300215379968000 z^{38}+7873377561448743273881600 z^{37}\\[6pt]
&\quad-12188753588700934348185600 z^{36}+11409105186984502777856000 z^{35}\\[6pt]
&\quad-1853370295840331059200000 z^{34}-12922596637778941349888000 z^{33}\\[6pt]
&\quad+19993842975085327602810880 z^{32}-4136695001339260651438080 z^{31}\\[6pt]
&\quad-40585258593920366687027200 z^{30}+107607975413970670190592000 z^{29}\\[6pt]
&\quad-189170246667253453894451200 z^{28}+290673733377906514130370560 z^{27}\\[6pt]
&\quad-429481500981884772899880960 z^{26}+614653426107799377123737600 z^{25}\\[6pt]
&\quad-825958110337598656348160000 z^{24}+1014095417844181497806848000 z^{23}\\[6pt]
&\quad-1125028176866670548300595200 z^{22}+1129311459482608004707123200 z^{21}\\[6pt]
&\quad-1033623338399676468559872000 z^{20}+869136778177466010173440000 z^{19}\\[6pt]
&\quad-672028063551221072396288000 z^{18}+473438441949368700161228800 z^{17}\\[6pt]
&\quad-299190013959544777788620800 z^{16}+167798468337926970277888000 z^{15}\\[6pt]
&\quad-84223564508812395151360000 z^{14}+39006701101726128144384000 z^{13}\\[6pt]
&\quad-16949659707832925998284800 z^{12}+6525804102142065953996800 z^{11}\\[6pt]
&\quad-1953358789335809261568000 z^{10}+408567853900785254400000 z^9\\[6pt]
&\quad-90672379909684330496000 z^8+43132985715615837716480 z^7\\[6pt]
&\quad-13837533253868380487680 z^6+78654993658072268800 z^5\\[6pt]
&\quad+776840149395832832000 z^4-482905506919219200 z^3\\[6pt]
&\quad-31960428074332323840 z^2-1612499772831170560 z\\[6pt]
&\quad-7036874417766400.
\end{align*}
The explicit  expression for  \eqref{diamond-eq2} is omitted.

We end this section by  noting that
our algorithmic approach can  be used to derive dissection formulas on  quotients in the form of \eqref{radu-gf}, that is,
\begin{align}\label{quo-radu}
  \prod_{\delta | M}
(q^\delta;q^\delta)^{r_\delta}_\infty,
\end{align}
where $M$ is a positive integer and $r_\delta$, $r_{\delta, g}$ are integers.
Let $a(n)$ be the partition function defined by \eqref{radu-gf},
and let $m$ be a positive integer. If our algorithm can be utilized to
find a formula for the generating function of $a(mn+t)$ for each $0\leq t \leq m-1$,
then we are led to an $m$-dissection formula on the quotient  \eqref{quo-radu}.
For example,
the algorithm is valid to produce the 5-dissection formulas
for $(q; q)_\infty$ and $\frac{1}{(q;q)_\infty}$, see
Berndt \cite[p. 165]{Berndt-book}.

\section{More General Partition Functions}\label{tran-law}

While many partition functions $a(n)$ are of the form \eqref{radu-gf},
there are  partition functions
that do not seem to fall into this framework, such as
Andrews' $(k,i)$-singular overpartition function $\overline{Q}_{k, i}(n)$.
Andrews \cite{Andrews-2015} derived the generating function
\begin{align}\label{asop}
\sum\limits_{n=0}^\infty \overline{Q}_{k, i}(n)q^n
= \frac{(q^k, -q^i, -q^{k-i}; q^k)_\infty}{(q; q)_\infty}.
\end{align}
In general, it is not always the case that
a quotient on the right hand side of \eqref{asop}
can be expressed in the form of \eqref{radu-gf}.

The objective of this section is to extend our algorithm
to  partition functions $b(n)$ defined by
\begin{align}\label{def-1}
\sum_{n=0}^{\infty}
b(n)q^n=\prod_{\delta | M}
(q^\delta;q^\delta)^{r_\delta}_\infty
\prod_{\delta|M\atop 0<g<\delta}
(q^g,q^{\delta-g};q^\delta)^{r_{\delta,g}}_\infty,
\end{align}
where $M$ is a positive integer and
$r_\delta$, $r_{\delta, g}$ are integers.
In fact,
for any $k$ and $1\leq i< \frac{k}{2}$,
\eqref{asop} can be written in the form of \eqref{def-1}:
\begin{align}\label{asop-1}
\sum\limits_{n=0}^\infty \overline{Q}_{k, i}(n)q^n
= \frac{(q^k; q^k)_\infty(q^{2i}, q^{2k-2i}; q^{2k})_\infty}{(q; q)_\infty
(q^i, q^{k-i}; q^k)_\infty},
\end{align}
where  $M=2k$,
\begin{align*}
r_\delta =
\begin{cases}
-1, &\delta = 1,\\
1, & \delta=k,\\
0, & \text{otherwise},
\end{cases}
\quad
\text{and}
\quad
r_{\delta, g} =
\begin{cases}
-1, &\delta=k, g=i,\\
1, & \delta=2k, g=2i,\\
0, & \text{otherwise}.
\end{cases}
\end{align*}

Analogous to the generating function
$g_{m,t}(\tau)$ in Sect. \ref{construction} as given by Radu \cite{Radu-2009},
we adopt the same notation for the generating function of $b(mn+t)$:
\begin{align}\label{gmt-2}
g_{m,t}(\tau)  = q^{\frac{t-\ell}{m}}\sum\limits_{n=0}^\infty b(mn+t)q^n,
\end{align}
where
\[\ell=-\frac{1}{24}\sum\limits_{\delta|M}\delta r_\delta-\sum\limits_{\delta|M \atop 0 < g<\delta}\frac{\delta}{2}P_2\left(\frac{g}{\delta}\right)r_{\delta, g}.\]
As before,
\[P_2(t)=\{t\}^2-\{t\}+\frac{1}{6},\]
and $\{t\}$ is the fractional part of $t$.

To derive a Ramanujan-type identity for $b(mn+t)$,
we follow the same procedure as  in Sect. \ref{Sketch}.
There are only a few modifications that should be taken into account
in order to
extend Theorem \ref{con_F_modular_function} and Theorem \ref{order_F-1}
to the generating function $g_{m,t}(\tau)$ in  \eqref{gmt-2}.
The proofs are similar to those of Theorem \ref{con_F_modular_function} and Theorem \ref{order_F-1} and hence are omitted.

Let $\phi(\tau)$ be a generalized eta-quotient and $F = \phi(\tau) g_{m,t}(\tau)$.
Similar to Theorems \ref{con_F_modular_function},
we give a criterion for $F(\tau)$ to be a modular function for $\Gamma_1(N)$.
Let $\kappa=\gcd(m^2-1, 24)$.
First, we assume that $N$ satisfies the following conditions:
\begin{description}\label{con-2}
\setlength{\parskip}{2ex}
  \item[{1}\label{con-2_N_1}] $M| N$.

  \item[{2}\label{con-2_N_2}] $p| N$ for any prime $p| m$.

  \item[{3}\label{con-2_N_3}] $\kappa N\sum\limits_{{\delta|M \atop 0<g<\delta}}\frac{g}{\delta}r_{\delta, g} \equiv 0 \pmod{2}$.

  \item[{4}\label{con-2_N_4}] $\kappa N\sum\limits_{{\delta|M\atop 0<g<\delta}}r_{\delta, g}\equiv 0 \pmod{4}$.

  \item[{5}\label{con-2_N_5}] $\kappa mN^2\sum\limits_{\delta|M\atop 0<g<\delta}\frac{r_{\delta, g}}{\delta} \equiv 0\pmod{12}$.

  \item[{6}\label{con-2_N_6}] $\kappa N\sum\limits_{\delta|M}r_\delta \equiv 0 \pmod{8}$.

  \item[{7}\label{con-2_N_7}] $\kappa mN^2\sum\limits_{\delta|M}\frac{r_\delta}{\delta}\equiv 0 \pmod{24}$.

  \item[{8}\label{con-2_N_8}] $\frac{24mM}{\gcd(\kappa \alpha(t), 24mM)}\left| N\right.$,
  where
  \begin{align*}
    \alpha(t) = -M\sum\limits_{\delta|M}\delta r_{\delta} -
  12M\sum\limits_{{\delta|M \atop 0 < g < \delta}}\delta P_2\left(\frac{g}{\delta}\right)r_{\delta,g}-24Mt.
  \end{align*}

  \item[{9}\label{con-2_N_9}] Let $\prod_{\delta|M}\delta^{|r_\delta|} = 2^zj$,
   where $z\in \mathbb{N}$ and $j$ is odd.
   If $2| m$, then $\kappa N\equiv 0 \pmod 4$ and $Nz\equiv 0 \pmod 8$,
    or $z\equiv 0 \pmod 2$ and $N(j-1)\equiv 0 \pmod 8$.

  \item[{10}\label{con-2_N_10}]
  Let $\mathbb{S}_n = \{j^2 \pmod n \colon j\in\mathbb{Z}_n,\  \gcd(j, n)=1,\  j\equiv 1 \pmod{N}\}$.
  For any $s\in \mathbb{S}_{24mM}$,
 \[\frac{s-1}{24}\sum\limits_{\delta|M}\delta r_\delta + (s-1)\sum\limits_{{\delta|M \atop 0<g<\delta}}\frac{\delta}{2}P_2\left(\frac{g}{\delta}\right)r_{\delta, g}+ts \equiv t \pmod m.\]
\end{description}
For a given partition function $b(n)$, and given
integers $m$ and $t$,  such a positive integer $N$ always exists, because $N=24mM$ satisfies the conditions \nameref{con-2_N_1}--\nameref{con-2_N_10}. For example,
for Andrews' (3,1)-singular overpartition function $\overline{Q}_{3, 1}(n)$,
and for  $m=9$ and $t=3$
we have $N=6$.
Compared with the conditions in Sect. \ref{construction},
the  conditions
\nameref{con-2_N_3}--\nameref{con-2_N_5}
are required to deal with the generalized eta-quotients.

\begin{thm}\label{con_F_modular_function-2}
For a given partition function $b(n)$ as defined by \eqref{def-1},
and for given integers $m$ and $t$,
suppose that $N$ is a positive integer satisfying the conditions \nameref{con-2_N_1}--\nameref{con-2_N_10}.
Let
\begin{align*}
  F(\tau)=\phi(\tau)\, g_{m,t}(\tau),
\end{align*}
where
\[\phi(\tau)=\prod_{\delta | N}\eta^{a_{\delta}}(\delta \tau)\,\prod_{{\delta|N\atop 0<g\leq \left\lfloor\delta/2\right\rfloor}}\eta_{\delta,g}^{a_{\delta,g}}(\tau),\]
and $a_{\delta}$ and $a_{\delta,g}$ are integers.
Then $F(\tau)$ is a modular function with respect to $\Gamma_1(N)$ if and only if $a_{\delta}$ and $a_{\delta,g}$ satisfy the following conditions:

\begin{enumerate}
\setlength{\parskip}{2ex}
  \item[{\rm(1)}\label{F_con_1}] $\sum\limits_{\delta|N}a_\delta+\sum\limits_{\delta|M}r_\delta=0$,

  \item[{\rm(2)}\label{F_con_2}] $N \sum\limits_{\delta|N}\frac{a_\delta}{\delta}
                  +2N\sum\limits_{{\delta|N\atop 0<g\leq \left\lfloor{\delta}/{2}\right\rfloor}}\frac{a_{\delta,g}}{\delta}
                  +Nm\sum\limits_{\delta|M}\frac{r_\delta}{\delta}
                  +2Nm \sum\limits_{{\delta|M\atop 0<g<{\delta}}}\frac{r_{\delta,g}}{\delta}\equiv0\pmod{24}$,

  \item[{\rm(3)}\label{F_con_3}]
  $\sum\limits_{\delta|N}\delta a_\delta
  +12 \sum\limits_{{\delta|N\atop 0<g\leq \left\lfloor{\delta}/{2}\right\rfloor}}\delta P_2\left(\frac{g}{\delta}\right){a_{\delta,g}}
  +m\sum\limits_{\delta|M}\delta r_\delta$\\[6pt]
  \rule{33pt}{0pt}$+ 12m \sum\limits_{{\delta|M\atop 0<g< {\delta}}}\delta P_2\left(\frac{g}{\delta}\right){r_{\delta,g}}
  +\frac{(m^2-1)\alpha(t)}{m M}
  \equiv0\pmod{24},$\\[6pt]
  where
  $$\alpha(t)=-M\sum\limits_{\delta|M}\delta r_{\delta} -
  12M\sum\limits_{{\delta|M \atop 0 < g < \delta}}\delta P_2\left(\frac{g}{\delta}\right)r_{\delta,g}-24Mt,$$

  \item[{\rm(4)}\label{F_con_4}] For any integer $0<a<12N$ with $\gcd{(a,6)}=1$ and $a\equiv 1\pmod N$,
  $$\prod\limits_{\delta|N}\left(\frac{\delta}{a}\right)^{|a_\delta|}
  \prod\limits_{\delta|M}
  \left(\frac{m\delta}{a}\right)^{|r_\delta|} e^{\sum\limits_{\delta|N}\sum\limits_{g=1}^{{\tiny \left\lfloor\delta/2\right\rfloor}}\pi i\big(\frac{g}{\delta}-\frac{1}{2}\big)(a-1)a_{\delta,g}+\sum\limits_{\delta|M}\sum\limits_{g=1}^{\delta-1}\pi i\big(\frac{g}{\delta}-\frac{1}{2}\big)(a-1)r_{\delta,g}}=1.$$
\end{enumerate}
\end{thm}

In the notation $p(\gamma, \lambda)$ and $p(\gamma)$ in \eqref{def_p_gamma-l-1} and \eqref{def_p_gamma-1}, we define
the map $p\colon \Gamma \times \mathbb{Z}_m \rightarrow \mathbb{Q}$
by
\begin{align*}
p(\gamma, \lambda) & =  \frac{1}{24}\sum\limits_{\delta|M}\frac{\gcd^2(\delta(a+\kappa\lambda c), mc)}{\delta m}r_\delta\\[6pt]
 &\qquad +\frac{1}{2}\sum\limits_{{\delta|M \atop 0<g<\delta}}\frac{\gcd^2(\delta(a+\kappa\lambda c), mc)}{\delta m}P_2\left(\frac{(a+\kappa\lambda c)g}{\gcd(\delta(a+\kappa\lambda c), mc)}\right)r_{\delta, g},
\end{align*}
and define $p(\gamma)$ by
\begin{eqnarray}\label{def_p_gamma}
p(\gamma) = \min\{p(\gamma, \lambda): \lambda=0, 1, \ldots, m-1\}.
\end{eqnarray}

Parallel to Theorem \ref{order_F-1},
we obtain lower bounds of the orders of $F(\tau)$ at cusps of $\Gamma_1(N)$.

\begin{thm}\label{order_F}
For a given partition function $b(n)$ as defined by \eqref{def-1},
and for given integers $m$ and $t$,
let
\begin{align*}
  F(\tau)=\phi(\tau)\,g_{m,t}(\tau),
\end{align*}
where
\[\phi(\tau)=\prod_{\delta | N}\eta^{a_{\delta}}(\delta \tau)\prod_{{\delta|N\atop 0<g\leq \left\lfloor\delta/2\right\rfloor}}\eta_{\delta,g}^{a_{\delta,g}}(\tau),\] $a_{\delta}$ and $a_{\delta,g}$ are integers.
Assume that $F(\tau)$ is a modular function for $\Gamma_1(N)$.
Let $\{s_1, s_2,\ldots,s_\epsilon\}$ be a complete set of inequivalent cusps of $\Gamma_1(N)$,
and for each $1\leq i \leq \epsilon$,
let $\alpha_i\in\Gamma$ be such that $\alpha_i\infty = s_i$.
Then
\begin{align}\label{ord_cusp}
\mathrm{ord}_{s_i}(F(\tau))\ge w_{\alpha_i}\,(p(\alpha_i)+p^*(\alpha_i)),
\end{align}
where $p(\gamma)$ is given by \eqref{def_p_gamma} and $p^*(\gamma)$ is defined in Lemma \ref{phi_trans}.
\end{thm}

For a given partition function $b(n)$, and given integers $m$ and $t$,
assume that we have found a generalized eta-quotient $\phi(\tau)$
such that
\begin{align}\label{F_expression-2}
F(\tau) = \phi(\tau)\,g_{m,t}(\tau)
\end{align}
is a modular function for $\Gamma_1(N)$.
Utilizing the algorithm  in Sect. \ref{Sketch},
we try to express $F(\tau)$ as a linear
combination of generalized eta-quotients with  level $N$.
If we succeed, then we obtain a Ramanujan-type identity for $b(mn+t)$.
Note that Theorem \ref{order_F} is needed to find
a generalized eta-quotient $h(\tau)$ such that $hF$ has a pole
only at infinity.

For example, we can derive   Ramanujan-type identities
 on the singular overpartition function introduced by Andrews \cite{Andrews-2015}.
The number of $(k,i)$-singular overpartitions of $n$ is denoted by $\overline{Q}_{k,i}(n)$ $(1\leq i< \frac{k}{2})$.
For $k=3$ and $i=1$,
 \eqref{asop-1} specializes to
\begin{align*}
  \sum_{n=0}^\infty \overline{Q}_{3,1}(n)q^n = \frac{(q^3;q^3)_\infty(q^2,q^4;q^6)_\infty}{(q;q)_\infty(q,q^2;q^3)_\infty}.
\end{align*}
When applied to the above generating function, our algorithm
 produces the Ramanujan-type identities on $\overline{Q}_{3,1}(9n+3)$ and $\overline{Q}_{3,1}(9n+6)$ due to Shen \cite{Shen-2016}.

\begin{thm}\label{THM-An-RTI}
We have
\begin{align*}
 \frac{(q;q)_{\infty}^{14}}
       {q(q^2;q^2)_{\infty}^5(q^3;q^3)_{\infty}^6(q^6;q^6)_{\infty}^{3}} \sum_{n=0}^\infty \overline{Q}_{3,1}(9n+3)q^n
=6z+96,
\end{align*}
and
\begin{align*}
\frac{(q;q)_{\infty}^{13}}
       {q(q^2;q^2)^4_{\infty}(q^3;q^3)^3_{\infty}(q^6;q^6)_{\infty}^{6}}
\sum_{n=0}^\infty \overline{Q}_{3,1}(9n+6)q^n=24 z+96,
\end{align*}
where
\begin{align*}
z = \frac{(q^2;q^2)_{\infty}^3(q^3;q^3)_{\infty}^9}
{q(q;q)_{\infty}^3(q^6;q^6)_{\infty}^9}.
\end{align*}
\end{thm}

Our extended algorithm can also be used to derive dissection formulas on the quotients in the form of \eqref{def-1}, that is,
\begin{align}\label{quo}
  \prod_{\delta | M}
(q^\delta;q^\delta)^{r_\delta}_\infty
\prod_{\delta|M\atop 0<g<\delta}
(q^g,q^{\delta-g};q^\delta)^{r_{\delta,g}}_\infty,
\end{align}
where $M$ is a positive integer and $r_\delta$, $r_{\delta, g}$ are integers.
Let $b(n)$ be the partition function defined by \eqref{def-1},
and let $m$ be a positive integer. If our algorithm can be utilized to
find a formula for the generating function of $b(mn+t)$ for each $0\leq t \leq m-1$,
then we are led to an $m$-dissection formula on the quotient in \eqref{quo}.
For example, we get the $2$-, $4$-dissections of the Rogers--Ramanujan continued fraction \cite{Ramanujan-1988, Andrew-1981,Hirschhorn-1998, Lewis-Liu-2000},
the 8-dissections of the Gordon's continued fraction \cite{Hirschhorn-2001, Xia-Yao-2011} and the $2$-, $3$-, $4$-, $6$-dissections of Ramanujan's cubic continued fraction \cite{Hirschhorn-Roselin-2010, Srivastava-2007}.

We now demonstrate how to deduce the $2$-dissection
formula for the Rogers--Ramanujan continued fraction:
$$R(q) = \frac{1}{1}\+\frac{q}{1}\+\frac{q^2}{1}\+\frac{q^3}{1}\+\dos.$$
Rogers  \cite[p. 329]{Rogers-1894} showed that
\begin{align}\label{R-R-c}
  R(q) = \frac{(q^2,q^3;q^5)_\infty}{(q,q^4;q^5)_\infty}.
\end{align}
The following $2$-dissection formulas of Ramanujan \cite[p. 50]{Ramanujan-1988}
were first proved by Andrews \cite{Andrew-1981}.
With respect to the quotient in \eqref{R-R-c}, we have to count on the extended algorithm
because \eqref{R-R-c} cannot be expressed in the form of \eqref{radu-gf}.

\begin{thm}
We have
  \begin{align}\label{R-R-1}
    R(q) = \frac{(q^8,q^{12};q^{20})_\infty^2}
    {(q^6,q^{14};q^{20})_\infty
     (q^{10},q^{10};q^{20})_\infty}
    +
    q\frac{(q^2,q^{18};q^{20})_\infty
    (q^8,q^{12};q^{20})_\infty}
    {(q^4,q^{16};q^{20})_\infty
     (q^{10},q^{10};q^{20})_\infty}
  \end{align}
  and
    \begin{align}\label{R-R-2}
    R(q)^{-1} = \frac{(q^4,q^{16};q^{20})_\infty^2}
    {(q^2,q^{18};q^{20})_\infty
     (q^{10},q^{10};q^{20})_\infty}
    -
    q\frac{(q^4,q^{16};q^{20})_\infty
    (q^6,q^{14};q^{20})_\infty}
    {(q^8,q^{12};q^{20})_\infty
     (q^{10},q^{10};q^{20})_\infty}.
  \end{align}
\end{thm}

\begin{proof}
 As far as  \eqref{R-R-c} is concerned, we have $M=5$,
  $r_{5,1}=-1$ and $r_{5,2}=1$.
  We find that $N=10$ satisfies the conditions \nameref{con-2_N_1}--\nameref{con-2_N_10}.
  Let $r(n)$ be defined by
  $$R(q) = \sum_{n=0}^\infty r(n) q^n. $$
  Employing our algorithm,
  we obtain that
  \begin{align*}
    \sum_{n=0}^\infty r(2n) q^n
    = \frac{z_1z_3}{z_2z^2}
    \cdot\frac{\eta^2_{10,5}(\tau)}{\eta^2_{10,4}(\tau)}
  \end{align*}
  and
  \begin{align*}
    \sum_{n=0}^\infty r(2n+1) q^n
    =
    \frac{z_2^3 z^4}{z_1^2 z_3^3}
    \cdot\frac{\eta^8_{10,4}(\tau)}{\eta^8_{10,5}(\tau)},
  \end{align*}
  where $z$, $z_1$, $z_2$ and $z_3$
  are given in \eqref{generator_10}.
A direct computation yields \eqref{R-R-1}.
Similarly, we get  \eqref{R-R-2}. This completes the proof.
\end{proof}

Gordon \cite{Gordon-1965} showed that
\begin{align}\label{Gordon-def}
  1+q+\frac{q^{2}}{1+q^{3}}\+ \frac{q^{4}}{1+q^{5}}\+ \frac{q^{6}}{1+q^{7}}\+\dos
  =\frac{(q^3,q^5;q^8)_\infty}{(q,q^7;q^8)_\infty}.
\end{align}

Using our algorithm, we deduce the
following $8$-dissection formulas of Hirschhorn for \eqref{Gordon-def} and its reciprocal, see \cite[pp. 373--374]{Hirschhorn-2001}.

\begin{thm}[Hirschhorn \cite{Hirschhorn-2001}]
We have
\begin{align}
\frac{\left(q^{3}, q^{5} ; q^{8}\right)_{\infty}}{\left(q, q^{7} ; q^{8}\right)_{\infty}}\notag
&=
\frac{\left(-q^{24},-q^{32},-q^{32},-q^{40}, q^{64}, q^{64} ; q^{64}\right)_{\infty}}{\left(q^{8}, q^{16}, q^{16}, q^{24}, q^{32}, q^{32} ; q^{32}\right)_{\infty}} \notag\\[6pt]
&\qquad+q \frac{\left(-q^{16},-q^{24},-q^{40},-q^{48}, q^{64}, q^{64} ; q^{64}\right)_{\infty}}{\left(q^{16}, q^{8}, q^{24}, q^{24}, q^{48}, q^{64}, q^{64} ; q^{64}\right)_{\infty}}\notag\\[6pt]
&\qquad +q^{2} \frac{\left(-q^{16},-q^{24},-q^{40},-q^{48}, q^{64}, q^{64} ; q^{64}\right)_{\infty}}{\left(q^{8}, q^{16}, q^{16}, q^{24}, q^{32}, q^{64} ; q^{64} ; q^{64}\right)_\infty}\notag\\[6pt]
&\qquad -2 q^{12} \frac{\left(-q^{8},-q^{16},-q^{64},-q^{64}, q^{64}, q^{64} ; q^{64}\right)_\infty}{\left(q^{8}, q^{16}, q^{16}, q^{24}, q^{32}, q^{32} ; q^{32}\right)_\infty}\notag\\[6pt]
&\qquad-q^5\frac{\left(-q^{8},-q^{16},-q^{48},-q^{56}, q^{64}, q^{64} ; q^{64}\right)_{\infty}}{\left(q^{8}, q^{8}, q^{24}, q^{24}, q^{32}, q^{32} ; q^{32}\right)_{\infty}} \notag\\[6pt]
&\qquad-q^{6} \frac{\left(-q^{8},-q^{16},-q^{48},-q^{56}, q^{64}, q^{64} ; q^{64}\right)_{\infty}}{\left(q^{8}, q^{16}, q^{16}, q^{24}, q^{32}, q^{32} ; q^{32}\right)_{\infty}},\notag\\[6pt]
\frac{\left(q, q^{7} ; q^{8}\right)_{\infty}}{\left(q^{3}, q^{5} ; q^{8}\right)_{\infty}}
&=\frac{\left(-q^{16},-q^{24},-q^{40},-q^{48}, q^{64}, q^{64} ; q^{64}\right)_{\infty}}{\left(q^{8}, q^{8}, q^{24}, q^{24}, q^{32}, q^{32} ; q^{32}\right)_{\infty}}\notag\\[6pt]
&\qquad -q \frac{\left(-q^{16},-q^{24},-q^{40},-q^{48}, q^{64}, q^{64} ; q^{64}\right)_{\infty}}{\left(q^{8}, q^{16}, q^{16}, q^{24}, q^{32}, q^{32} ; q^{32}\right)_{\infty}}
\notag\\[6pt]
&\qquad +q^{3} \frac{\left(-q^{8},-q^{32},-q^{32},-q^{56}, q^{64}, q^{64} ; q^{64}\right)_{\infty}}{\left(q^{8}, q^{16}, q^{16}, q^{24}, q^{32}, q^{32} ; q^{32}\right)_{\infty}}
\notag\\[6pt]
&\qquad -q^{4} \frac{\left(-q^{8},-q^{16},-q^{48},-q^{56}, q^{64}, q^{64} ; q^{64}\right)_{\infty}}{\left(q^{8}, q^{8}, q^{24}, q^{24}, q^{32}, q^{32} ; q^{32}\right)_{\infty}}\notag\\[6pt]
&\qquad + q^{5} \frac{\left(-q^{8},-q^{16},-q^{48},-q^{56}, q^{64}, q^{64} ; q^{64}\right)_{\infty}}{\left(q^{8}, q^{16}, q^{16}, q^{24}, q^{32}, q^{32} ; q^{32}\right)_{\infty}}\notag\\[6pt]
&\qquad-2 q^{7} \frac{\left(-q^{24},-q^{40},-q^{64},-q^{64}, q^{64}, q^{64} ; q^{64}\right)_{\infty}}{\left(q^{8}, q^{16}, q^{16}, q^{24}, q^{32}, q^{32} ; q^{32}\right)_{\infty}}.\notag
\end{align}
\end{thm}

Ramanujan's cubic continued fraction is defined by
\begin{align*}
  \frac{1}{1}\+\frac{q+q^{2}}{1}\+\frac{q^{2}+q^{4}}{1}\+\dos,
\end{align*}
which equals
\begin{align}\label{R-C}
\frac{\left(q, q^{5} ; q^{6}\right)_{\infty}}{\left(q^{3}, q^{3} ; q^{6}\right)_{\infty}},
\end{align}
see \cite[p. 44]{Ramanujan-1988}.
Applying our algorithm to \eqref{R-C} and its reciprocal,
we are led to the
$2$-, $3$-, $4$- and $6$-dissection formulas in Theorem 1.1--Theorem 1.4 in \cite{Hirschhorn-Roselin-2010}.

\section*{Acknowledgements}

We are grateful to Peter Paule for his inspiring lectures and for stimulating discussions. We would also like to thank the referees for their valuable comments and suggestions.


\begin{thebibliography}{99}

\bibitem{4ti2}
4ti2 team:
4ti2 -- A software package for algebraic, geometric and combinatorial problems on linear spaces software.
Available at \url{https://4ti2.github.io}

\bibitem{Andrew-1981}
Andrews, G.E.:
Ramunujan's ``lost'' notebook III. The Rogers--Ramanujan continued fraction.
Adv. Math. 41,  186--208 (1981)

\bibitem{Andrews-2015}
Andrews, G.E.:
Singular overpartitions.
Int. J. Number Theory 11(5), 1523--1533 (2015)

\bibitem{Andrews-Paule-2007}
Andrews, G.E., Paule, P.:
MacMahon's partition analysis XI: Broken diamonds and modular forms.
Acta Arith. 126(3), 281--294  (2007)

\bibitem{Atkin-Swinnerton-Dyer-1954}
Atkin, A.O.L., Swinnerton-Dyer, P.:
Some properties of partitions.
Proc. London Math. Soc. (3) 4, 84--106 (1954)

\bibitem{Berndt-book}
Berndt, B.C.:
Number Theory in the Spirit of Ramanujan.
Student Mathematical Library, 34.
Amer. Math. Soc., Providence, RI (2006)

\bibitem{Bilgici-Ekin-2014-13}
Bilgici, G., Ekin, A.B.:
Some congruences for modulus 13 related to partition generating function.
Ramanujan J. 33(2), 197--218 (2014)

\bibitem{Bilgici-Ekin-2014}
Bilgici, G., Ekin, A.B.:
$11$-Dissection and modulo $11$ congruences properties for partition generating function.
Int. J. Contemp. Math. Sci. 9(1-4), 1--10  (2014)

\bibitem{Chan-2008}
Chan, S.H.:
Some congruences for Andrews--Paule's broken 2-diamond partitions. Discrete Math. 308(23), 5735--5741 (2008)

\bibitem{Cho-Koo-Park-2009}
Cho, B., Koo, J.K., Park, Y.K.: Arithmetic of the Ramanujan--G\"{o}llnitz--Gordon continued fraction. J.
Number Theory 129(4), 922--947 (2009)

\bibitem{Corteel-Lovejoy}
Corteel, S., Lovejoy, J.: Overpartitions. Trans. Amer. Math. Soc. 356(4), 1623--1635 (2004)


\bibitem{Diamond-Shurman-2005}
Diamond, F., Shurman, J.:
A First Course in Modular Forms.
Graduate Texts in Mathematics, 228. Springer-Verlag, New York (2005)

\bibitem{Eichhorn-1999}
Eichhorn, D.A.:
Some results on the congruential and gap-theoretic study of partition functions.
Ph.D. Thesis. University of Illinois at Urbana-Champaign (1999)

\bibitem{Eichhorn-Ono-1995}
Eichhorn, D.A., Ono, K.: Congruences for partition functions.
In: Berndt, B.C., Diamond, H.G., Hildebrand, A.J. (eds.)
Analytic Number Theory, Vol. 1 (Allerton Park, IL, 1995), pp. 309--321.
Progr. Math., 138, Birkh\"{a}user Boston, Boston, MA (1996)

\bibitem{Eichhorn-Sellers-2002}
Eichhorn, D.A., Sellers, J.A.:
Computational proofs of congruences for 2-colored Frobenius partitions.
Int. J. Math. Math. Sci. 29(6), 333--340 (2002)

\bibitem{Gasper-Rahman-2004}
Gasper, G., Rahman, M.:
Basic Hypergeometric Series.
Second Edition. Encyclopedia of Mathematics and its Applications, 96. Cambridge University Press, Cambridge (2004)

\bibitem{Gordon-1965}
Gordon, B.:
Some continued fractions of the Rogers--Ramanujan type.
Duke Math. J. 32, 741--748 (1965)


\bibitem{Hardy-1927}
Hardy, G.H.:
Note on Ramanujan's arithmetic function $\tau(n)$.
Proc. Cambridge Philos. Soc. 23, 675--680 (1927)

\bibitem{Hardy-1938}
Hardy, G.H.:
A further note on Ramanujan's arithmetic function $\tau(n)$.
Proc. Cambridge Philos. Soc. 34, 309--315 (1938)




\bibitem{Hemmecke-2018}
Hemmecke, R.: Dancing samba with Ramanujan partition congruences. J. Symbolic Comput. 84, 14--24 (2018)

\bibitem{Hirschhorn-1998}
Hirschhorn, M.D.:
On the expansion of Ramanujan's continued fraction.
Ramanujan J. 2(4),  521--527 (1998)

\bibitem{Hirschhorn-2001}
Hirschhorn, M.D.:
On the expansion of a continued fraction of Gordon.
Ramanujan J. 5(4), 369--375 (2001)


\bibitem{Hirschhorn-Roselin-2010}
Hirschhorn, M.D., Roselin:
On the 2-, 3-, 4- and 6-dissections of Ramanujan's cubic continued fraction and its reciprocal.
In: Baruah, N.D., Berndt, B.C.,  Cooper, S., Huber,  T.,  Schlosser, M.J. (eds.)
Ramanujan Rediscovered, pp. 125--138.
Ramanujan Math. Soc. Lect. Notes Ser., 14, Ramanujan Math. Soc., Mysore (2010)

\bibitem{Hirschhorn-Sellers}
Hirschhorn, M.D., Sellers, J.A.: Arithmetic relations for overpartitions. J. Combin. Math. Combin. Comput. 53, 65--73 (2005)



\bibitem{Knopp-1993}
Knopp, M.:
Modular Functions in Analytic Number Theory. Second Edition.
Amer. Math. Soc., Chelsea Publishing (1993)


\bibitem{Kolberg-1957}
Kolberg, O.: Some identities involving the partition function. Math. Scand. 5, 77--92 (1957)

\bibitem{Lewis-Liu-2000}
Lewis, R., Liu, Z.-G.:
A conjecture of Hirschhorn on the 4-dissection of Ramanujan's continued fraction.
Ramanujan J. 4(4), 347--352 (2000)

\bibitem{Newman-1957}
Newman, M.:
Construction and application of a class of modular functions.
Proc. London. Math. Soc. (3) 7, 334--350 (1957)

\bibitem{Newman-1959}
Newman, M.:
Construction and application of a class of modular functions (\uppercase\expandafter{\romannumeral2}).
Proc. Lond. Math. Soc. (3) 9, 373--387 (1959)

\bibitem{Paule-Radu-2019-congruence}
Paule, P., Radu, C.-S.:
A unified algorithmic framework for Ramanujan's congruences
modulo powers of 5, 7, and 11. Preprint (2018)

\bibitem{Paule-Radu-2017-witness}
Paule, P., Radu, C.-S.:
A new witness identity for $11|p(11n + 6)$.
In: Andrews, G.E., Garvan, F. (eds.)
Analytic Number Theory, Modular Forms and $q$-Hypergeometric Series,
Springer Proc. Math. Stat., 221, pp. 625--639. Springer, Cham (2017)


\bibitem{Peter-Radu-2017-2}
Paule, P., Radu, S.:
Partition analysis, modular functions, and computer algebra.
In: Beveridge, A., Griggs, J.R., Hogben, L., Musiker, G., Tetali, P. (eds.)
Recent Trends in Combinatorics,
IMA Vol. Math. Appl., 159,  pp. 511--543.
Springer, Cham (2016)

\bibitem{Rademacher-1942}
Rademacher, H.:
The Ramanujan identities under modular substitutions.
Trans. Amer. Math. Soc. 51, 609--636 (1942)

\bibitem{Rademacher-1973}
Rademacher, H.:
Topics in Analytic Number Theory.
Die Grundlehren der mathematischen Wissenschaften, Band 169.
Springer-Verlag,
New York-Heidelberg (1973)

\bibitem{Radu-2015}
Radu, C.-S.: An algorithmic approach to Ramanujan--Kolberg identities. J. Symbolic Comput. 68(1), 225--253 (2015)

\bibitem{Radu-Thesis}
Radu, S.:
An algorithmic approach to Ramanujan's congruences and related problems.
Ph.D. Thesis. Research Institute for Symbolic Computation
Johannes Kepler University, Linz (2009)

\bibitem{Radu-2009}
Radu, S.:
An algorithmic approach to Ramanujan's congruences.
Ramanujan J. 20(2), 215--251 (2009)

\bibitem{Ramanujan-1919}
Ramanujan, S.:
Some properties of $p(n)$, the number of partitions of $n$. Proc. Cambridge Philos. Soc. 19, 207--210 (1919)

\bibitem{Ramanujan-1916}
Ramanujan, S.:
On certain arithmetical functions.
Trans. Cambridge Philos. Soc. 22,  159--184 (1916)

\bibitem{Ramanujan-1988}
Ramanujan, S.:
The Lost Notebook and Other Unpublished Papers.
Narosa Publishing House, New Delhi (1988)


\bibitem{Robins-1994}
Robins, S.:
Generalized Dedekind $\eta$-products.
In: Andrews, G.E., Bressoud, D.M., Parson, L.A. (eds.)
The Rademacher Legacy to Mathematics (University Park, PA, 1992),
Contemp. Math., 166, pp. 119--128.
Amer. Math. Soc., Providence, RI (1994)

\bibitem{Rogers-1894}
Rogers, L.J.:
Second memoir on the expansion of certain infinite products.
Proc. London Math. Soc. 25, 318--343 (1894)

\bibitem{Schoeneberg-1974}
Schoeneberg, B.:
Elliptic Modular Functions: An Introduction.
Die Grundlehren der mathematischen Wissenschaften,
Band 203. Springer-Verlag,
New York-Heidelberg (1974)

\bibitem{Schrijver-1986}
Schrijver, A.:
Theory of Linear and Integer Programming.
Wiley-Interscience Series in Discrete Mathematics.
John Wiley \& Sons, Ltd., Chichester (1986)


\bibitem{Shen-2016}
Shen, E.Y.Y.:
Arithmetic properties of $l$-regular overpartitions.
Int. J. Number Theory 12(3), 841--852 (2016)

\bibitem{Smoot-2019}
Smoot, N.A.:
An implementation of Radu's Ramanujan--Kolberg algorithm.
RISC     Technical Report (2019)

\bibitem{Srivastava-2007}
Srivastava, B.:
On 2-dissection and 4-dissection of Ramanujan's cubic continued fraction and identities.
Tamsui Oxf. J. Math. Sci. 23(3), 305--315 (2007)

\bibitem{Stein-Modular}
Stein, W.:
Modular Forms, A Computational Approach.
Graduate Studies in Mathematics, 79.
Amer. Math. Soc., Providence, RI (2007)

\bibitem{Xia-Yao-2011}
Xia, E.X.W., Yao, X.M.:
The 8-dissection of the Ramanujan--G\"{o}llnitz--Gordon continued fraction by an iterative method.
Int. J. Number Theory 7(6), 1589--1593 (2011)


\bibitem{Zuckerman-1939}
Zuckerman, H.S.:
Identities analogous to Ramanujan's identities involving the partition function.
Duke Math. J. 5(1), 88--110 (1939)

\end{thebibliography}
\end{document}